\documentclass[11pt,reqno]{amsart}
\usepackage{a4wide}
\usepackage{hyperref}
\usepackage{color}
\usepackage{cite}
\usepackage{amsfonts,amssymb}
\usepackage{pb-diagram}
\usepackage[all]{xy}
\usepackage{graphicx,color}
\usepackage{mathrsfs}
\usepackage{amstext}
\usepackage{sseq}

\definecolor{green}{rgb}{0,0.6,0.2}
\definecolor{blue}{rgb}{0,0,1}

\hypersetup{colorlinks,
linkcolor=blue,
filecolor=green,
citecolor=green}

\newtheorem{neu}{}[section]
\newtheorem{Cor}[neu]{Corollary}
\newtheorem*{Cor*}{Corollary}
\newtheorem{Thm}[neu]{Theorem}
\newtheorem*{Thm*}{Theorem}
\newtheorem{Prop}[neu]{Proposition}
\newtheorem*{Prop*}{Proposition}
\theoremstyle{definition}
\newtheorem{Lemma}[neu]{Lemma}
\newtheorem*{Rmk*}{Remark}
\newtheorem{Rmk}[neu]{Remark}

\newtheorem*{Ex*}{Example}

\newtheorem*{Qu*}{Question}

\newtheorem{Def}[neu]{Definition}

\theoremstyle{remark}

\theoremstyle{definition}

\newcommand{\x}{\times}

\newcommand{\wt}{\widetilde}
\newcommand{\wh}{\widehat}

\newcommand{\p}{\partial}
\newcommand{\om}{\omega}
\newcommand{\Om}{\Omega}
\newcommand{\into}{\hookrightarrow}
\newcommand{\pf}{\longrightarrow}

\newcommand{\N}{{\mathbb{N}}}
\newcommand{\Z}{{\mathbb{Z}}}
\newcommand{\R}{{\mathbb{R}}}
\newcommand{\C}{{\mathbb{C}}}

\renewcommand{\H}{\mathbb{H}}

\newcommand{\dom}{\mathrm{dom\,}}

\newcommand{\Id}{\mathrm{Id}}
\newcommand{\im}{\mathrm{im}}

\newcommand{\Sp}{\mathrm{Sp}}
\newcommand{\wind}{\mathrm{wind}}
\newcommand{\Fix}{\mathrm{Fix}\,}

\newcommand{\Sym}{{\rm Sym}}
\newcommand{\Hom}{{\rm Hom}}
\newcommand{\Ind}{{\rm Ind}}

\newcommand{\FS}{{\rm FS}}
\newcommand{\CZ}{{\rm CZ}}
\newcommand{\RS}{{\rm RS}}

\renewcommand{\SS}{\mathcal{S}}

\newcommand{\II}{\mathcal{I}}

\newcommand{\ZZ}{\mathcal{Z}}
\newcommand{\JJ}{\mathcal{J}}
\newcommand{\DD}{\mathcal{D}}

\newcommand{\UU}{\mathcal{U}}

\newcommand{\CC}{\mathcal{C}}

\newcommand{\MM}{\mathcal{M}}

\newcommand{\beq}{\begin{equation}}
\newcommand{\beqn}{\begin{equation}\nonumber}
\newcommand{\eeq}{\end{equation}}

\newcommand{\bea}{\begin{equation}\begin{aligned}}
\newcommand{\bean}{\begin{equation}\begin{aligned}\nonumber}
\newcommand{\eea}{\end{aligned}\end{equation}}

\newcommand{\one}
{{{\mathchoice \mathrm{ 1\mskip-4mu l} \mathrm{ 1\mskip-4mu l}
\mathrm{ 1\mskip-4.5mu l} \mathrm{ 1\mskip-5mu l}}}}

\numberwithin{equation}{section}
\numberwithin{figure}{section}

\begin{document}
\title[Real holomorphic curves and invariant global surfaces of section]{Real holomorphic curves and \\invariant global surfaces of section}
\author{Urs Frauenfelder  and Jungsoo Kang}
\address{
	Urs Frauenfelder\\
	Institut f\"ur Mathematik\\
	Universit\"at Augsburg, Universit\"atsstrasse 14, D-86159 Augsburg, Germany}
\email{urs.frauenfelder@math.uni-augsburg.de}
\address{
	Jungsoo Kang\\ 
	Mathematisches Institut\\
	Westf\"alische Wilhelms-Universit\"at M\"unster\\
	Einsteinstrasse 62, D-48149 M\"unster, Germany }
\email{jungsoo.kang@me.com}
\maketitle
\begin{abstract}
In this paper we prove that a dynamically convex starshaped hypersurface in $\mathbb{C}^2$
which is invariant under complex conjugation admits a global surface of section which is invariant
under conjugation as well. We obtain this invariant global surface by embedding $\mathbb{C}^2$ into
$\mathbb{CP}^2$ and applying a stretching argument to real holomorphic curves in $\mathbb{CP}^2$.
The motivation for this result arises from recent progress in applying holomorphic curve techniques to
gain a deeper understanding on the dynamics of the restricted three body problem. 
\end{abstract}
\setcounter{tocdepth}{1}

\tableofcontents

\section{Introduction}

A global  surface of section is a gadget which allows one to reduce the Hamiltonian dynamics on a three dimensional energy hypersurface to the study of an area preserving disk map. An influential result
due to Hofer, Wysocki, and Zehnder \cite{HWZ98} tells us that on a dynamically convex, starshaped hypersurface in
$\mathbb{C}^2$ a global disk-like surface of section always exists. Here we recall that dynamically convex means
that the Conley-Zehnder indices of all periodic orbits are at least 3. For example hypersurfaces which bound
a strictly convex domain are dynamically convex, but different than convexity the notion of dynamical convexity
is a symplectic notion. The main result of this paper tells us that under the assumptions of the theorem of Hofer, Wysocki and Zehnder if our hypersurface is additionally invariant under complex conjugation we can choose
the global surface of section invariant as well. 

The study of this paper is motivated by recent progress of applying methods from holomorphic curve
theory to the study of the dynamics of the restricted three body problem, see \cite{AFFHvK12}. In particular,
in \cite{AFFHvK12} it was shown that below the first critical value for sufficient small mass of the smaller of the two primaries a global surface of section exists in the component around the small body. 

An interesting aspect of the Hamiltonian of the restricted three body problem is that it is invariant under an
antisymplectic involution, \cite{Poi,Bir15}. This leads to the dichotomy of periodic orbits into symmetric and nonsymmetric ones, i.e. periodic orbits whose trace is invariant under the involution respectively orbits whose
trace is not invariant and which have to occur necessarily in pairs. 

The Levi-Civita regularization embeds a double cover of the energy hypersurface of the restricted three body problem in complex two dimensional vector space. The results from \cite{AFPvK12} imply that
below the first critical value the components of the energy hypersurface around the two primaries bound
after Levi-Civita regularization a starshaped domain in complex two dimensional vector space. The fact that
the restricted three body problem is invariant under an antisymplectic involution translates into the fact
that the starshaped domains are invariant under complex conjugation. 

The main result of this paper asserts that if the boundary of an invariant starshaped domain is in addition dynamically convex, then it carries an invariant global disk-like surface of section. The results from \cite{AFFHvK12}
imply that this happens for example close to the smaller primary below the first critical value, if the mass
of the smaller primary is small enough. We refer to  \cite{Kan14} for a study of the nontrivial dynamical implications the existence of an invariant global disk-like surface of section has. For example the existence of an invariant global surface of section implies that there are either two or infinitely many symmetric periodic orbits and 
if there is at least one nonsymmetric periodic orbit, there have to exist necessarily infinitely many symmetric ones. 

In the paper \cite{HWZ03} Hofer, Wysocki, and Zehnder applied techniques from holomorphic curve theory combined with methods from Symplectic Field theory to produce a global disk-like surface of section via a stretching argument. The method of proof of our result is inspired by the stretching argument in \cite{HWZ03}. But different from
\cite{HWZ03} we apply technology from open string theory instead of closed string theory, namely real holomorphic curves, which are also used to define Welschinger invariants \cite{Wel05}. 

Since real holomorphic curves have to be defined via antiinvariant almost complex structures which in general
do not satisfy the generic properties of general almost complex structures, our approach requires some new ideas. Our main source of inspiration is the theory of fast finite energy planes due to Hryniewicz \cite{Hry12}. 

The problem due to lack of genericity of antiinvariant almost complex structures is that we are not able to
guarantee that the binding orbit of our invariant global surface of section has Conley-Zehnder index 3. 
On the other hand our binding orbit is a symmetric periodic orbit. Now a symmetric periodic orbit
has features from closed and open string theory. In fact it can be interpreted as a periodic orbit as well
as a Lagrangian intersection point. Due to this fact, a symmetric periodic orbit carries two indices. We show in this paper how its index as a Lagrangian intersection point can be controlled. We give an interpretation of this
second index in terms of winding numbers, inspired from the paper \cite{HWZ95b}. This enables us to conclude that the invariant finite energy plane we obtain by the stretching method is fast in the sense of Hryniewicz.
Because the real holomorphic curves we used in the stretching process were embedded, the invariant fast finite
energy plane we obtain is embedded as well. A deep theorem of Hryniewicz based on the fundamental results by Hofer, Wysocki, and Zehnder than allows us to conclude
that its projection to the contact manifold is an invariant global disk-like surface of section in case our contact manifold
is dynamically convex. 

The two indices of a symmetric periodic orbit are not completely unrelated but their difference can be computed
as a H\"ormander index. This allows us to conclude that the Conley-Zehnder index of the binding orbit
of our invariant global surface of section is at most 4. We could not decide so far, if it is possible to find
always an invariant global surface of section whose binding orbit has Conley-Zehnder index 3 and therefore
leave this question as food for thought for future research.

\subsubsection*{Acknowledgments} {We are grateful to Peter Albers, Urs Fuchs, Chris Wendl, and Kai Zehmisch for helpful conversations. We also thank the referee for his/her thorough work. JK is supported by DFG grant KA 4010/1-1 and was partially supported by SFB 878-Groups, Geometry, and Actions during this project.}


\section{Statement of the results}
On the complex two dimensional plane $\mathbb{C}^2$ we consider the standard symplectic form 
$$
\omega=dx_1 \wedge dy_1+dx_2 \wedge dy_2,
$$ 
where $(z_1,z_2)=(x_1+iy_1,x_2+iy_2)$ are coordinates on $\mathbb{C}^2$. The vector field 
$$
L=\frac{1}{2}\Big(x_1\frac{\partial}{\partial x_1}+y_1\frac{\partial}{\partial y_1}+x_2\frac{\partial}{\partial x_2}+y_2\frac{\partial}{\partial y_2}\Big)
$$
is a Liouville vector field with respect to $\omega$. A closed hypersurface $M \subset \mathbb{C}^2$ is called \emph{starshaped} if the Liouville vector field $L$ intersects it transversally. In particular, the one-form
$$
\lambda:=\iota_L \omega
$$
endows $M$ a contact form $\alpha:=\lambda|_M$, whose associated contact structure $\xi=\ker\alpha$ coincides with the tight contact structure on $M$. Using the flow of the Liouville vector field $L$ we identify $\mathbb{C}^2 \setminus \{0\}$ with the symplectization $(\R\x M,d(e^r\alpha))$ of the hypersurface $(M,\alpha)$  where $r$ in the coordinate on $\R$ . The {\em Reeb vector field} $X$ is defined on $M$ by the requirement
$$
\alpha(X)=1, \quad \iota_X d\alpha=0
$$
and extended $\mathbb{R}$-invariant to the symplectization $\R\x M$. An $\omega|_\xi$-compatible almost complex structure $J:\xi\to\xi$  is also extended to an almost complex structure $\wt J$ on $\R\x M$  by keeping the hyperdistribution $\xi$ invariant and satisfying $\wt JL=X$. Such a \emph{SFT-like} almost complex structure $\wt J$ is $\R$-invariant and  $\om$-compatible.  Abbreviate 
$$
\CC:=\{\phi \in C^\infty(\mathbb{R},[0,1])\,|\,\phi' \geq 0\}.
$$ 
For $\phi \in\CC$ define $\alpha_\phi \in \Omega^1(\mathbb{R}\times M)$ by $\lambda_\phi(r,x)=\phi(r)\alpha(x)$, $(r,x)\in\R\x M$ and set $\omega_\phi=d\alpha_\phi$. The \emph{energy} $E(\tilde{u}) \in \R$ of a map $\tilde{u}:\C\to\R\x M$ is defined as
$$
E(\tilde{u})=\sup_{\phi \in \CC} \int_\mathbb{C} \tilde{u}^* \omega_\phi.
$$
A map $\tilde{u}:(\C,i)\to (\R\x M,\wt J) $ is called a {\bf finite energy plane} if it is (pseudo-) holomorphic, i.e. $T\tilde u\circ i=\wt J\circ T\tilde u$ and satisfies
$$
0<E(\tilde{u})<\infty.
$$
The theory of finite energy planes was intensely studied by Hofer, Wysocki, and Zehnder in 
\cite{HWZ95b, HWZ96, HWZ96b, HWZ99}. 
In the following, let us assume that our hypersurface $M$ is \emph{nondegenerate}, in the sense that all periodic orbits of the Reeb vector field $X$ on $M$ have precisely one Floquet multiplier equal to 1.  Under this assumption, it follows from the results in \cite{HWZ96},  that if $\tilde{u}=(a,u):\C\to\R\x M$ is a finite energy plane, then $u(e^{2\pi(s+it)})$ converges in
$C^\infty(\mathbb{R}/\mathbb{Z},M)$ to $P(Tt)$ as $s \to \infty$, where $P \in C^\infty([0,T],M)$ is a periodic orbit of the Reeb vector field $X$, i.e.\,a solution of the problem
$$
\dot P(t)=X(P(t)), \qquad P(0)=P(T).
$$

\begin{Def}
A finite energy plane $\tilde{u}=(a,u):\C\to\R\x M$ is called {\bf fast}, if its asymptotic periodic orbit is nondegenerate and simple, and if $u$ is an immersion and transversal to the Reeb vector field $X$.
\end{Def}

The notion of fast finite energy planes is due to Hryniewicz \cite{Hry12}. This terminology is explained from the fact that a fast finite energy plane decays asymptotically as fast as possible. The reason why fast finite energy planes are interesting, is that they are essential ingredients to construct global disk-like surfaces of section.\\[-1.5ex]

In the following we consider the antisymplectic involution
$$
\tilde\rho \colon \mathbb{C}^2 \to \mathbb{C}^2, \quad (z_1,z_2) \mapsto (\bar{z}_1,\bar{z}_2).
$$
Note that the Liouville vector field $L$ is invariant under $\tilde\rho$, i.e.\,$T\tilde\rho(L)=L$, so that $\tilde\rho^*\lambda=-\lambda$. This restricts to an  antisymplectic involution on  the symplectization $\R\x M\cong \C^2\setminus\{0\}$ of $M$ that we still denote by $\tilde\rho:\R\x M\to \R\x M$. If $M$ is invariant under $\tilde\rho$, $\rho:=\tilde\rho|_M$ is an involution on $M$ with $\rho^*\alpha=-\alpha$ and  $\tilde\rho=\Id_\R\x\rho$ on $\R\x M$. Note that $T\rho$ maps $\xi$ to itself. We denote by $\JJ_\rho$ the space of $\om|_\xi$-compatible almost complex structures on $\xi$ antiinvariant under $\rho$, explicitly 
$$
 \rho^*J:=(T\rho|_\xi)^{-1}\circ J \circ T\rho|_\xi =-J.
$$ 
For any $J\in\JJ_\rho$, the associated SFT-like almost complex structure $\wt J$ is $\wt\rho$-antiinvariant.
We say that a finite energy plane $\tilde{u}:(\C,i)\to(\R\x M,\wt J)$ for $J\in\JJ_\rho$ is {\bf invariant}, if it satisfies
$$
\tilde\rho \circ \tilde{u}=\tilde{u}\circ I
$$
where 
$$
I:\C\to\C,\quad z\mapsto \bar z.
$$
\begin{Def}
Let $M$ be a 3-dimensional smooth manifold with a smooth vector field $X$. A global surface of section for $X$ on $M$ is an embedded Riemann surface $\Sigma\into M$ meeting the following requirements.\\[-2ex]
\begin{itemize}
\item[i)] The boundary of $\Sigma$ consists of periodic orbits, called the {\em spanning orbits}.
\item[ii)] The  vector field $X$ is transversal to the interior $\mathring\Sigma$ of $\Sigma$.
\item[iii)] Every orbit of $X$, except the spanning orbits, passes through $\mathring\Sigma$ in forward and backward time.\\[-2ex]
\end{itemize}
If in addition $M$ carries a smooth involution $\rho$, a global surface of section $\Sigma$ invariant under $\rho$ is called an {\bf invariant global surface of section}. 
\end{Def}

Throughout this paper, we are interested in starshaped hypersurfaces $M \subset \mathbb{C}^2$ invariant under $\tilde\rho$ with the Reeb vector field $X$ on $M$.  A starshaped hypersurface $M \subset \mathbb{C}^2$ is called {\bf dynamically convex}, if the Conley-Zehnder index of every periodic orbit of  $X$ is at least 3. If $M\subset \C^2$ bounds a strictly convex domain, it is dynamically convex, see \cite{HWZ98,Lon02}. Note that the notion of dynamical convexity is a symplectic invariant notion, while the notion of convexity is not. We mention the fact that there are examples of energy hypersurfaces for the restricted three body problem which are dynamically convex, although their Levi-Civita embedding is not convex, see \cite{AFFvK13}.

\begin{Thm}\label{thm:invglobal}
Assume that $M \subset \mathbb{C}^2$ is a starshaped hypersurface invariant under $\tilde\rho$ such that $(M,X)$ in $\C^2$ is dynamically convex. Then $M$ has an invariant global disk-like surface of section $\DD$ for the Reeb flow with the following properties.
\begin{itemize}
\item[1.] The Conley-Zehnder index of the spanning orbit of $\DD$ is either 3 or 4;
\item[2.] $d\lambda|_{\mathring\DD}$ is symplectic and the area $\int_{\mathring\DD}d\lambda$ is identical to the period of the spanning orbit;
\item[3.] The Poincar\'e return map $f:(\mathring\DD,d\lambda|_{\mathring\DD})\to(\mathring\DD,d\lambda|_{\mathring\DD})$ is symplectic and the involution $\rho|_{\mathring\DD}:(\mathring\DD,d\lambda|_{\mathring\DD})\to(\mathring\DD,d\lambda|_{\dot\DD})$ is antisymplectic.
\end{itemize}
\end{Thm}

Theorem \ref{thm:invglobal} is proved in Section 6. The existence of an invariant global disk-like surface of section  immediately implies that there is a symmetric open book decomposition of $M$ of which every page is a global disk-like surface of section,  see Remark \ref{rmk:symmetric open book decomposition}. In consequence we can find two invariant global disk-like surfaces of section. In fact, it is possible to construct a symmetric holomorphic open book decomposition by following a standard Fredholm theory in \cite{HWZ99,Hry12} but we decided not to include this because we could not find any advantage of holomorphicity in the present situation. However if $(M,X)$ carries a dihedral symmetry, we expect that this is needed to construct an open book decomposition which respects the dihedral symmetry. The dynamical convexity assumption can be weakened as in \cite{HS11}. Moreover we expect that our idea can be used to construct a symmetric finite energy foliation in the absence of dynamical convexity, see \cite{HWZ03}.

 In the following example borrowed from \cite[Lemma 1.6]{HWZ95a}, we can see a symmetric open book decomposition and two invariant global disk-like surfaces of section in ellipsoids.

\begin{Rmk}[{\em Symmetric open book decomposition of ellipsoids}] Consider an ellipsoid
$$
E=\bigg\{(z_1,z_2)\in\C^2\,\bigg|\,\frac{|z_1|^2}{r_1^2}+\frac{|z_2|^2}{r_2^2}=1\bigg\}.
$$
An $S^1$-family of holomorphic planes
$$
\tilde u_\theta:\C\to \R\x E,\quad\theta\in S^1
$$
where $\theta$ is the angular coordinate on $\C$ is defined by
$$
\tilde u_\theta(z)=(a_\theta(z),u_\theta(z))=\bigg(\frac{1}{2}\log\sqrt{|z|^2+1},\frac{(r_1z,r_2\theta)}{\sqrt{|z|^2+1}}\bigg).
$$
Therefore we have a open book decomposition of $E$ of which pages are given by
$$
\DD_\theta=\big\{(z_1,z_2)\in E\,\big|\, z_2\in \theta\R_+ \;\;\textrm {with}\;\;|z_2|\in[0,r_2]\big\}
$$
where $\theta\R_+=\{a\theta|a\in\R_+\}$. Moreover every page is a global disk-like surface of section for the Reeb vector field $X$ on $E$. In particular, $\DD_1$ and $\DD_{-1}$ are global disk-like surfaces of section invariant under $\rho$.
\end{Rmk}

\noindent{\em Reversibility of the Poincar\'e return map.} In the presence of an invariant global disk-like surface of section $(\DD,\rho|_{\DD})$, to find a symmetric periodic orbit is reduced to find a symmetric periodic point of the Poincar\'e return map $f$ on $(\mathring\DD,\rho|_{\mathring\DD})$. A point $x\in\DD$ is called {\em symmetric periodic point} if 
$$
f^k(x)=x,\quad f^\ell(x)=\rho(x)\quad\textrm{for some }\, k,\,\ell\in\N.
$$ 
Due to $\rho^*X=-X$,  $f$ obeys the {\em reversible condition}, namely 
$$
f\circ\rho|_{\mathring\DD}=\rho|_{\mathring\DD}\circ f^{-1}.
$$
Concerning the study of symmetric periodic points of reversible maps, we refer to \cite{Kan14}. In particular, it was proved that every area preserving map on the open unit disk in the complex plane reversible with respect to a reflection about an axis has at least one symmetric fixed point, i.e. $k=\ell=1$, and that there are infinitely many symmetric periodic points provided another (possibly nonsymmetric) periodic point. Note that the Poincar\'e return map and the involution on $\mathring\DD$ is smoothly conjugate to an area preserving map and the reflection on the open unit disk, see \cite[Section 5]{HWZ98}. Since the spanning orbit of $\DD$ is symmetric, we have the following consequence which refines  \cite[Theorem 1.1]{HWZ98}. 
\begin{Thm} 
Let $(M,X,\rho)$ be as in Theorem \ref{thm:invglobal}. There are either two or infinitely many symmetric periodic orbits. Furthermore, a nonsymmetric periodic orbit ensures infinitely many symmetric periodic orbits.
\end{Thm}

Denote by $P'$ a symmetric periodic orbit which corresponds to a symmetric fixed point of the Poincar\'e return map $f$ on $(\mathring\DD,\rho|_{\mathring\DD})$ guaranteed by \cite{Kan14}. The following result follows immediately from Theorem 1.2 and Lemma 4.13 in \cite{Hry11}.
\begin{Thm}
The symmetric periodic orbit $P'$ is a spanning orbit of another invariant global disk-like surface of section.
\end{Thm}

\section{Indices for symmetric periodic orbits}
In \cite{HWZ95b}, Hofer, Wysocki, and Zehnder introduced an alternative equivalent definition of the Conley-Zehnder index using winding numbers of eigenvalues of a certain self-adjoint operator. This allows us to link the asymptotic behavior of a finite energy surface to the Conley-Zehnder indices of the asymptotic periodic orbit, see Section 3. In order for the study of invariant finite energy surfaces, we will prove analogous results for the Robbin-Salamon index and establish a relation between the Conley-Zehnder index of a symmetric periodic orbit and its Robbin-Salamon index.

\subsection{Robbin-Salamon index}
We denote by $\mathrm{Lag}(\R^{2n})$ the Grassmanian manifold of Lagrangian subspaces in $\R^{2n}$ with the standard symplectic structure $\om_{\textrm{std}}=d\mathbf{x}\wedge d\mathbf{y}$. For a path of Lagrangian subspaces $\Lambda:[0,T]\to\mathrm{Lag}(\R^{2n})$ and a single Lagrangian $V\in\mathrm{Lag}(\R^{2n})$, Robbin and Salamon in \cite{RS93} define a Maslov-type index 
$$
\mu_\RS(\Lambda,V)\in\frac{1}{2}\Z.
$$
This takes values in $\Z+\frac{1}{2}$ if $n=1$, $\Lambda(0)=V$, and $\Lambda(T)\cap V=\{0\}$.  Among many other properties, the Robbin-Salamon index has the following properties, which will be used in
the sequel. 
\begin{itemize}
\item[1.]({\em Maslov}) If $\Lambda:[0,T]\to\mathrm{Lag}(\R^{2n})$ is a loop, i.e. $\Lambda(0)=\Lambda(T)$, $\mu_\RS(\Lambda,V)$ agrees with the Maslov index. In particular it is independent of the choice of $V$.
\item[2.]({\em Reversal}) For $\Lambda:[0,T]\to\mathrm{Lag}(\R^{2n})$,
$$
\mu_\RS(\Lambda,V)=-\mu_\RS(\overline\Lambda,V)
$$
where $\overline\Lambda(t):=\Lambda(T-t)$.
\item[3.]({\em Naturality}) If $\Gamma:[0,T]\to\Sp(\R^{2n})$ is a path of symplectic matrices, for $V_1,\,V_2\in\mathrm{Lag}(\R^{2n})$,
$$
\mu_\RS(\Gamma V_1,V_2)=-\mu_\RS(\Gamma^{-1} V_2,V_1).
$$
In contrast if $\Gamma$ is a path of antisymplectic matrices,
$$
\mu_\RS(\Gamma V_1,V_2)=\mu_\RS(\Gamma^{-1} V_2,V_1).
$$
\item[4.]({\em Homotopy}) For $\Lambda_i:[0,T_i]\to\mathrm{Lag}(\R^{2n})$, $i=1,\,2$ with the same ends  $\Lambda_1(0)=\Lambda_2(0)$ and $\Lambda_1(T_1)=\Lambda_2(T_2)$, if $\Lambda_1$ and $\Lambda_2$ are homotopic relative to their end points,
$$
\mu_\RS(\Lambda_1,V)=\mu_\RS(\Lambda_2,V).
$$
\item[5.] ({\em Catenation}) For $\Lambda_i:[0,T_i]\to\mathrm{Lag}(\R^{2n})$, $i=1,\,2$ with $\Lambda_1(T_1)=\Lambda_2(0)$,
$$
\mu_\RS(\Lambda_1,V)+\mu_\RS(\Lambda_2,V)=\mu_\RS(\Lambda_2\#\Lambda_1,V)
$$
where the operation $\#$ joins two paths so that $\Lambda_2\#\Lambda_1:[0,T_1+T_2]\to\mathrm{Lag}(\R^{2n})$. 
\end{itemize}

If $\Psi:[0,T]\to\Sp(\R^{2n})$ is a path of symplectic matrices with $\Psi(0)=\one$ and $\det(\Psi(T)-\one_{\R^{2n}})\neq0$, the {\bf Conley-Zehnder index} of $\Psi$ is defined by
$$
\mu_{\CZ}(\Psi):=\mu_{\RS}(\overline{\mathrm{gr}}(\Psi),\overline{\Delta})
$$
where $\overline{\mathrm{gr}}(\Psi)=\{(x,-\Psi(x)\,|\,x\in\R^{2n}\}$ and $\overline{\Delta}=\{(x,-x)\,|\,x\in\R^{2n}\}$ which are Lagrangian in $(\R^{2n}\x\R^{2n},\om_\textrm{std}\oplus \om_\textrm{std})$. This coincides with the original definition of the Conley-Zehnder index, see \cite{CZ84,SZ92}. 

The following propositions for loops of Lagrangians will be useful later.

\begin{Prop}\label{Prop:loop property of RS-index}
Let $V,\,\Lambda(t)\in\mathrm{Lag}(\R^{2n})$ for $t\in[0,T]$ with $\Lambda(0)=V$. If $\Gamma:[0,T]\to\Sp(\R^{2n})$ such that $\Gamma(0)=\one_{\R^{2n}}$ and $\Gamma(T)V=V$, we have
$$
\mu_{\RS}(\Gamma\Lambda,V)=\mu_{\RS}(\Lambda,V)+\mu_{\RS}(\Gamma V,V).
$$
\end{Prop}
\begin{proof}
We first note that the two Lagrangian paths $\Gamma(t)\Lambda(t)$ and $(\Gamma(t)\Lambda(T))\#\Lambda$ are homotopic with fixed ends. Indeed, the map $F:[0,1]\x[0,T]\to\mathrm{Lag}(\R^{2n})$ given by
$$
F(s,t)=\left\{\begin{array}{ll} \Lambda(2t)\quad& t\leq \frac{sT}{2},\\[.5ex]
\Gamma\Big(\frac{2}{2-s}\big(t-\frac{sT}{2}\big)\Big)\Lambda\Big(s+\frac{2(1-s)}{2-s}\big(t-\frac{sT}{2}\big)\Big) &t\geq\frac{sT}{2}
\end{array}\right.
$$
and a reparametrization of $F(1,t)$ provide a homotopy between them. By the homotopy property and the catenation property, we have
$$
\mu_{\RS}(\Gamma\Lambda,V)=\mu_{\RS}(\Lambda,V)+\mu_{\RS}(\Gamma\Lambda(T),V).
$$
We can rewrite the last term using the naturality property and the Maslov property as follows since  $\Gamma^{-1}V$ is a loop of Lagrangian subspaces by the assumption.
$$
\mu_{\RS}(\Gamma\Lambda(T),V)=-\mu_\RS(\Gamma^{-1}V,\Lambda(T))=-\mu_\RS(\Gamma^{-1}V,V)=\mu_\RS(\Gamma V,V).
$$
This proves the proposition.
\end{proof}

\begin{Prop}\label{Prop:loop property of RS-index 2}
Let $I_V$ be an antisymplectic involution on $\R^{2n}$ with $\Fix I_V=V\in\mathrm{Lag}(\R^{2n})$ and let $\Gamma:[0,T]\to\Sp(\R^{2n})$ satisfying
$$
I_V\Gamma(t)I_V=\Gamma(T-t).
$$
Then we have
$$
\mu_\RS(\Gamma V,V)=2\mu_\RS(\Gamma_1 V,V)
$$
where $\Gamma_1:=\Gamma|_{[0,\frac{T}{2}]}:[0,\frac{T}{2}]\to\Sp(\R^{2n})$.
\end{Prop}
\begin{proof}
Abbreviate $\Gamma_2(t):=\Gamma|_{[\frac{T}{2},T]}(t+\frac{T}{2}):[0,\frac{T}{2}]\to\Sp(\R^{2n})$ so that $\Gamma=\Gamma_1\#\Gamma_2$.
Then by the catenation property we have
$$
\mu_\RS(\Gamma V,V)=\mu_\RS(\Gamma_1 V,V)+\mu_\RS(\Gamma_2 V,V)
$$
since  $\Gamma(\frac{T}{2})V=V$ due to the assumption. Again by the assumption,
$$
\Gamma_2(t)=I_V\overline\Gamma_1(t)I_V
$$
where $\overline\Gamma_1(t):=\Gamma_1(\frac{T}{2}-t)$, and the proposition follows from
$$
\mu_\RS(\Gamma_2 V,V)=\mu_\RS(I_V\overline\Gamma_1I_V V,V)=\mu_\RS(\Gamma_1 V,V).
$$
\end{proof}
\subsection{Winding numbers}
Let $J_0$ be the standard complex structure on $\R^2$:
$$
J_0=\left(\begin{array}{cc} 0 & -1\\
1 & 0  \end{array}\right).
$$
Abbreviate $S^1_T=\R/T\Z$. For a loop of symmetric matrices 
$$
S:\R\to\mathrm{Sym}(\R^2),\quad S(t+T)=S(T),
$$ 
we define the unbounded self-adjoint operator on $L^2(S^1_T,\R^2)$ by 
\bea\label{eq:asymptotic operator}
A_S=-J_0 \frac{d}{dt}-S(t)
\eea
with domain $\dom A_S= W^{1,2}(S_T^1,\R^2)$. Since $W^{1,2}(S_T^1,\R^2)\into L^2(S_T^1,\R^2)$ is compact, $A_S$ is an operator with compact resolvent. Therefore the spectrum $\sigma(A_S)$ of $A_S$ consists entirely of  isolated real eigenvalues of $A_S$ which accumulate only at $\pm\infty$. Moreover, the multiplicity of the eigenvalues is at most 2. Assume that a nonzero element $\gamma\in L^2(S_T^1,\R^2)$ is an eigenfunction with respect to the eigenvalue $\lambda\in\R$. That is, $\gamma$ is a solution of the first order differential equation
$$
-J_0\dot\gamma(t)-S(t)\gamma(t)=\lambda \gamma(t),
$$
and thus $\gamma(t)\neq 0$ for all $t\in S^1_T$. Therefore we can associate to the eigenfunction $\gamma$ of $A_S$  the {\em winding number} defined by
$$
w(\gamma,\lambda,S):=\frac{1}{2\pi}\big[(\arg(\gamma(T))-\arg(\gamma(0))\big]\in\Z.
$$
We recall the nontrivial properties of the winding number from \cite{HWZ95b}.

\begin{Lemma}[\!\!\cite{HWZ95b}]\label{Lemma:properties of wind} Let $S$ and $A_S$ be as above. 
\begin{itemize}
\item[1.] If $\beta$ and $\gamma$ are eigenfunctions belonging to the same eigenvalue $\lambda$, 
$$
w(\beta,\lambda,S)=w(\gamma,\lambda,S).
$$
Thus we set $w(\lambda,S):=w(\gamma,\lambda,S)$ for any $\gamma$ in the eigenspace of $\lambda$.
\item[2.] For every $k\in\Z$,
$$
\#\{\lambda\in\R\,|\, w(\lambda, S)=k,\; \lambda\in\sigma(A_S)\}=2.
$$
\item[3.] For any two eigenvalues $\lambda$ and $\mu\in\sigma(A_S)$ with $\lambda\leq\mu$,
$$
w(\lambda,S)\leq w(\mu,S).
$$
\end{itemize}
\end{Lemma}
We denote the maximal winding number among the negative eigenvalues by
$$
\alpha(S):=\max\big\{w(\lambda,S)\,\big|\,\lambda\in\sigma(A_S)\cap(-\infty,0)\big\}.
$$
The properties of the winding number in the above lemma imply that the following number $p(S)$ is either 0 or 1.
$$
p(S):=\min\big\{w(\lambda,S)\,\big|\,\lambda\in\sigma(A_S)\cap[0,\infty)\big\}-\alpha(S)\in\{0,1\}.
$$
We define a Maslov-type index of a loop of symmetric matrices $S:S_T^1\to\Sym(\R^2)$ by
$$
\mu(S):=2\alpha(S)+p(S)\in\Z.
$$
If we arrange the eigenvalues of $A_S$ by $\lambda_i^j(S)$, $j\in\{0,1\}$, $i\in\Z$ so that 
\beq\label{eq:arrangement of eigenvalues I}
w(\lambda_i^j,S)=i,\qquad \lambda_i^0(S)\leq\lambda_i^1(S),
\eeq
the index is rewritten by
$$
\mu(S)=\max\{2i+j\,|\,\lambda_i^j(S)<0\}.
$$

For each $\Psi:\R\to\Sp(\R^2)$ with $\Psi(0)=\one$ and $\Psi(t+T)=\Psi(t)\Psi(T)$, we associate the path of symmetric matrices
$$
S_\Psi(t):=-J_0\dot\Psi(t)\Psi(t)^{-1}
$$
satisfying 
$$
S_\Psi(t+T)=S_\Psi(t),\quad t\in\R.
$$
It turns out in \cite{HWZ95b} that if $\det(\Psi(T)-\one_{\R^{2}})\neq0$,
$$
\mu_\CZ(\Psi|_{[0,T]})=\mu(S_\Psi|_{[0,T]}).\\[2ex]
$$

Our next task is to relate the Robbin-Salamon index to the (half-) winding numbers of eigenvalues of a certain self-adjoint operator in a similar vein as above. Abusing notation we use the same letter to denote
$$
I=\left(\begin{array}{cc} 1 & 0\\
0 & -1  \end{array}\right):\R^2\to\R^2
$$
which is consistent with the original definition $I(z)=\bar z$, $z\in\C$ via the canonical identification. We consider the following Sobolev space of paths with boundary conditions.
\bean
W^{1,2}_I([0,\tfrac{T}{2}],\R^2)&=\{v\in W^{1,2}([0,\tfrac{T}{2}],\R^2)\,|\, v(0),\,v(\tfrac{T}{2})\in \R\}
\eea
where $\R$ stands for the real axis in $\R^2$. We also denote by $J_0\R$ the imaginary axis in $\R^2$. Note that $\R=\Fix I$ and $J_0\R=\Fix (-I)$.
Let $D(t)$, for $t\in[0,\frac{T}{2}]$, be symmetric matrices in $\R^2$ such that $D(0)$ and $D(\frac{T}{2})$ are diagonal matrices. As before we associate to $D$ the unbounded self-adjoint operator $A_D$ on $L^2([0,\frac{T}{2}],\R^2)$ with domain $\dom A_D= W_I^{1,2}([0,\frac{T}{2}],\R^2)$ defined by
\bea\label{eq:asymptotic operator with boundary condition}
A_D=-J_0\frac{d}{dt}-D(t).
\eea
The spectrum $\sigma(A_D)$ of $A_D$ consists of real eigenvalues of $A_D$, is discrete, and accumulates only at $\pm\infty$. Any nonzero eigenfunction $\gamma\in L^2([0,\frac{T}{2}],\R^2)$ of $A_D$ belonging to an eigenvalue $\lambda\in\R$, i.e.  
$$
-J_0\dot \gamma(t)-D(t)\gamma(t)=\lambda \gamma(t),
$$
is of class $C^1$ and never zero for all $t\in[0,\frac{T}{2}]$. Moreover we observe that  $\dot \gamma(0),\, \dot \gamma(\tfrac{T}{2})\in J_0\R$ since $D(0)$ and $D(\frac{T}{2})$ are diagonal. We define the {\em relative winding number} $w(\gamma,\lambda,D)$ as follows.
$$
w(\gamma,\lambda,D)=\frac{1}{2\pi}\big[\arg(\gamma(\tfrac{T}{2}))-\arg(\gamma(0))\big]\in\frac{1}{2}\Z.
$$
In what follows we show some properties of the relative winding number corresponding to the properties of the winding number in Lemma \ref{Lemma:properties of wind}.
\begin{Lemma}
If $\beta$ and $\gamma$ are nonzero eigenfunctions of $A_D$ corresponding to the same eigenvalue $\lambda$, they are linearly dependent, i.e.,
$$
\beta(t)=\tau \gamma(t),\quad \tau\in\R\setminus\{0\}.
$$
\end{Lemma}
\begin{proof}
We choose $\tau\in\R\setminus\{0\}$ so that $\beta(0)=\tau\gamma(0)$. Since $\kappa(t):=\beta(t)-\tau\gamma(t)$ ia also an eigenfunction with $\kappa(0)=0$, $\kappa\equiv 0$, and hence $\beta(t)=\tau \gamma(t)$.
\end{proof}

Thanks to the previous lemma, the relative winding number depends only on the eigenvalues of $A_D$ and thus we set
$$
w(\lambda,D):=w(\gamma,\lambda,D)
$$
where $\gamma$ is any eigenfunction of $A_D$ belonging to the eigenvalue $\lambda$. 

\begin{Lemma}
For each $k\in\Z$, there exists a unique eigenvalue $\lambda$ of $A_D$ satisfying 
$$
w(\lambda,D)=\frac{k}{2}.
$$
\end{Lemma}
\begin{proof}
This follows from Kato's perturbation theory \cite{Kat76} and the fact that $A_D$ is a bounded perturbation of the operator $-J_0\frac{d}{dt}$, cf. \cite[Lemma 3.6]{HWZ95b}.
\end{proof}
According to the lemma, we can arrange the spectrum of $A_D$ by
$$
\sigma(A_D)=\{\lambda_k(D)\}_{k\in\Z}
$$
where $\lambda_k(D)$ is characterized by
\beq\label{eq:arrangement of eigenvalues II}
w(\lambda_k,D)=\frac{k}{2}.
\eeq
Note that $\lambda_k<\lambda_{k+1}$, $k\in\Z$, cf. \cite[Lemma 3.6]{HWZ95b}. As before we abbreviate the maximal relative winding number  occurring for negative eigenvalues of $A_D$ by
$$
\alpha_I(D)=\max\{w(\lambda,D)\,|\,\lambda\in\sigma(A_D)\cap(-\infty,0)\}.
$$
We define a relative Maslov-type index by
$$
\mu_I(D):=2\alpha_I(D)+\frac{1}{2}\in\Z+\frac{1}{2}.
$$
Moreover, we can rewrite as 
$$
\mu_I(D)=\max\{k\in\Z\,|\,\lambda_k(D)<0\}+\frac{1}{2}.
$$

For $\Psi:\R\to\Sp(\R^2)$ with $\Psi(0)=\one$, $\Psi(t+T)=\Psi(t)\Psi(T)$, and $\Psi(-t)=I\Psi(t)I$, the associated loop of symmetric matrices
$$
D_\Psi(t):=-J_0\dot\Psi(t)\Psi(t)^{-1},\quad D_\Psi(t+T)=D_\Psi(t),\quad t\in\R.
$$
additionally satisfies
$$
D_\Psi(-t)=I D(t) I,\quad t\in\R.
$$
In particular, $D(\frac{mT}{2})$, $m\in\Z$ are diagonal matrices. $D_\Psi$ is the same as $S_\Psi$ by definition but we often use $D_\Psi$ to emphasize that $D_\Psi(\frac{mT}{2})$s are diagonal. Checking the characterizing axioms of the Robbin-Salamon index \cite[Theorem 4.1]{RS93}, one can show that 
$$
\mu_I(D_\Psi|_{[0,\frac{T}{2}]})=\mu_\RS(\Psi|_{[0,\frac{T}{2}]}\R,\R)
$$
provided $\Psi(\frac{T}{2})\R\cap \R=\{0\}$, see \eqref{eq:nondegeneracy}. Note that $v(t)\in\ker A_{D_\Psi}$ if and only if $v(t)=\Psi(t)v_0$ with $v_0\in \R\cap\Psi(\frac{T}{2})^{-1}\R$.

On the other hand, one can associate the Conley-Zehnder index $\mu_\CZ(\Psi|_{[0,T]})$ for a given path $\Psi$. To figure out the difference between $\mu_\CZ(\Psi|_{[0,T]})$ and $\mu_\RS(\Psi|_{[0,\frac{T}{2}]}\R,\R))$, we need another auxiliary index. Consider
\bean
W^{1,2}_{-I}([0,\tfrac{T}{2}],\R^2)&=\big\{v\in W^{1,2}([0,\tfrac{T}{2}],\R^2)\,\big|\, v(0),\,v(\tfrac{T}{2})\in J_0\R\big\}.
\eea
Let $D_\Psi(t)$ be as above, i.e. a path of symmetric matrices associated to $\Psi:\R\to\Sp(\R^2)$ with $\Psi(0)=\one$, $\Psi(t+T)=\Psi(t)\Psi(T)$, and $\Psi(-t)=I\Psi(t) I$. We can define an unbounded self-adjoint operator 
$$
\bar A_{D_{\Psi}|_{[0,\frac{T}{2}]}}: L^2([0,\tfrac{T}{2}],\R^2)\to L^2([0,\tfrac{T}{2}],\R^2),\quad v\mapsto -J_0\dot v -D_\Psi v
$$ 
with $\dom \bar A_{D_{\Psi}|_{[0,\frac{T}{2}]}}=W^{1,2}_{-I}([0,\frac{T}{2}],\R^2)$. As above we are able to define a relative winding number $w(\bar\lambda,D_{\Psi}|_{[0,\frac{T}{2}]})\in\frac{1}{2}\Z$ for each eigenvalue $\bar\lambda\in\R$ of $\bar A_{D_{\Psi}|_{[0,\frac{T}{2}]}}$. Also there is a unique eigenvalue $\bar\lambda_k$ of $\bar A_{D_{\Psi}|_{[0,\frac{T}{2}]}}$ with 
\beq\label{eq:arrangement of eigenvalues III}
w(\bar\lambda_k,D_{\Psi}|_{[0,\frac{T}{2}]})=\frac{k}{2},\quad k\in\Z.
\eeq
We define 
$$
\mu_{-I}(D_{\Psi}|_{[0,\frac{T}{2}]}):=\max\{k\in\Z\,|\,\bar\lambda_k(D_{\Psi}|_{[0,\frac{T}{2}]})<0\}+\frac{1}{2}.
$$
One can check again that
$$
\mu_{-I}(D_{\Psi}|_{[0,\frac{T}{2}]})=\mu_{\RS}(\Psi|_{[0,\frac{T}{2}]} J_0\R,J_0\R).
$$

\subsection{Symmetric periodic orbits and indices}

Although the following proposition is stated for dimension 2 since we use the indices $\mu_{-I}$, $\mu_I$, and $\mu$ in the proof, it holds in higher dimensions as well. Long, Zhang, and Zhu \cite{LZZ06} showed this for different (but equivalent up to constant) Maslov-type indices for every dimension. 

\begin{Prop} \label{prop:index relation}
If $\Psi:\R\to\Sp(\R^2)$ with $\Psi(0)=\one$, $\Psi(t+T)=\Psi(t)\Psi(T)$, and $\Psi(-t)=I\Psi(t) I$,
$$
\mu_{\CZ}(\Psi|_{[0,T]})=\mu_{\RS}(\Psi|_{[0,\frac{T}{2}]} \R,\R)+\mu_{\RS}(\Psi|_{[0,\frac{T}{2}]} J_0\R,J_0\R).
$$
\end{Prop}
\begin{proof}
It suffices to show that 
$$
\mu(S_{\Psi}|_{[0,T]})=\mu_I(S_{\Psi}|_{[0,\frac{T}{2}]})+\mu_{-I}(S_{\Psi}|_{[0,\frac{T}{2}]})
$$
where $S_\Psi:\R\to\mathrm{Sym}(\R^2)$ is a $T$-periodic loop of symmetric matrices associated to $\Psi$ as above. We consider the orthogonal decomposition
$$
L^2(S^1_T,\R^2)=L_1\oplus L_2,\quad W^{1,2}(S^1_T,\R^2)=W_1\oplus W_2
$$
where
$$
L_1:=\{v\in L^2(S^1_T,\R^2)\,|\,v(-t)=Iv(t)\},\quad
L_2:=\{v\in L^2(S^1_T,\R^2)\,|\,v(-t)=-Iv(t)\}.
$$
and 
$$
W_1:=W^{1,2}(S^1_T,\R)\cap L_1,\quad
W_2:=W^{1,2}(S^1_T,\R)\cap L_2.
$$
Indeed for any $v=(v_1,v_2)\in L^2(S^1,\R^2)$, the above decomposition is given by $v=\Upsilon_1+\Upsilon_2$ where 
$$
\Upsilon_1(t):=\bigg(\frac{v_1(t)+v_1(-t)}{2},\frac{v_2(t)-v_2(-t)}{2}\bigg)\in L_1
$$
and 
$$
\Upsilon_2(t):=\bigg(\frac{v_1(t)-v_1(-t)}{2},\frac{v_2(t)+v_2(-t)}{2}\bigg)\in L_2.
$$
We choose a homotopy $\wt S:[0,1]\x\R\to \mathrm{Sym}(\R^2)$ such that 
$$
\wt S(s,t+T)=\wt S(s,t),\quad I\wt S(s,t) I=\wt S(s,-t),\quad  \wt S(0,t)\equiv0,\quad  \wt S(1,t)=S_\Psi(t).
$$
Consider a path of operators $A_{\wt S}$ on $L^2(S^1_T,\R^2)$. 
Since $S_\Psi(s,-t)=I S_\Psi(s,t) I$,
$$
A_{\wt S}|_{L_1}:W_1\subset L_1\to L_1,\quad A_{\wt S}|_{L_2}:W_1\subset L_2\to L_2
$$
and thus 
$$
A_{\wt S}=A_{\wt S}|_{L_1}\oplus A_{\wt S}|_{L_2}.
$$
In consequence, we have
$$
\sigma(A_{\wt S})=\sigma(A_{\wt S}|_{L_1})\cup\sigma(A_{\wt S}|_{L_2}).
$$
Observe that 
$$
W^{1,2}_I([0,\tfrac{T}{2}],\R^2)=\{v|_{[0,\frac{T}{2}]}\,|\,v\in W_1\},\quad W^{1,2}_{-I}([0,\tfrac{T}{2}],\R^2)=\{v|_{[0,\frac{T}{2}]}\,|\,v\in W_2\}.
$$
A path of operators $A_{\wt S(s,t)}|_{L_1}$ can be seen as being defined on $W^{1,2}_I([0,\frac{T}{2}],\R^2)$ by restricting domain of paths. This restriction does not change the spectrum of $A_{\wt S(s,t)}|_{L_1}$. The same holds for $A_{\wt S(s,t)}|_{L_2}$ and $W^{1,2}_{-I}([0,\frac{T}{2}],\R^2)$.
We note that when $s=0$, i.e. $A_0=-J_0\frac{\p}{\p t}$,
$$
\lambda_i(A_0|_{W^{1,2}_I([0,\frac{T}{2}],\R^2)})=\lambda_{i}^0(A_0),\quad  \lambda_i(A_0|_{W^{1,2}_{-I}([0,\frac{T}{2}],\R^2)})=\lambda^1_{i}(A_0).
$$
for $i\in\Z$ where $\lambda_i$, $\lambda_i^0$, and $\lambda_i^j$ are defined in \eqref{eq:arrangement of eigenvalues I},  \eqref{eq:arrangement of eigenvalues II}, and \eqref{eq:arrangement of eigenvalues III} respectively. Here $\lambda^0_{i}(A_0)=\lambda^1_{i}(A_0)$. Due to Kato's perturbation theory of the eigenvalues of self-adjoint operators in \cite{Kat76} (or see \cite{HWZ95b}), the path of eigenvalues $\lambda_{i}^j(A_{\wt S})$ is continuous in $s\in[0,T]$ and $\lambda_{i}^j(A_{\wt S})$ and $\lambda_{i'}^{j'}(A_{\wt S})$ meet only if $i=i'$. Thus,
$$
\big\{\lambda_{i}^0(A_{\wt S}),\lambda_{i}^1(A_{\wt S})\big\}=\big\{\lambda_i(A_{\wt S}|_{W_I^{1,2}([0,\frac{T}{2}],\R^2)}),\lambda_i(A_{\wt S}|_{W_{-I}^{1,2}([0,\frac{T}{2}],\R^2)})\big\}
$$
for $i\in\Z$ and $s\in[0,T]$. This yields that
$$
\max\{2i+j\,|\,\lambda_i^j(A_{\wt S})<0\}=\max\{i\,|\,\lambda_i(A_{\wt S}|_{W_I^{1,2}})<0\}+\max\{i\,|\,\lambda_i(A_{\wt S}|_{W^{1,2}_{-I}})<0\}+1
$$
which in turn shows
$$
\mu(S_{\Psi}|_{[0,T]})=\mu_I(S_{\Psi}|_{[0,\frac{T}{2}]})+\mu_{-I}(S_{\Psi}|_{[0,\frac{T}{2}]}).
$$
\end{proof}

In higher dimensions, we denote by
$$
I_{\R^{2n}}=\left(\begin{array}{cc} \one_{\R^n} & 0\\
0 & -\one_{\R^n}  \end{array}\right):\R^{2n}\to\R^{2n},\quad \R^n=\Fix I_{\R^{2n}},\quad   J_0\R^n=\Fix (-I_{\R^{2n}}).
$$
We say that a path $\Psi:[0,T]\to\Sp(\R^{2n})$ with $\Psi(0)=\one_{\R^{2n}}$ is {\bf nondegenerate} if 
$$
\dim\ker\big(\Psi(T)-\one_{\R^{2n}}\big)=0,
$$
and that a pair of a path of Lagrangians and a single Lagrangian $(\Psi(t) V,V)$ for $\Psi:[0,T]\to\Sp(\R^{2n})$, $\Psi(0)=\one_{\R^{2n}}$ and $V\in\mathrm{Lag}(\R^{2n})$ is {\bf nondegenerate} if 
\beq\label{eq:nondegeneracy}
\dim\big(\Psi(T) V\cap V\big)=0.
\eeq
\begin{Prop}
Suppose that $\Psi:\R\to\Sp(\R^2)$ satisfies $\Psi(0)=\one_{\R^{2n}}$, $\Psi(t+T)=\Psi(t)\Psi(T)$, and $\Psi(-t)=I_{\R^{2n}}\Psi(t) I_{\R^{2n}}$. A path $\Psi|_{[0,T]}$ is nondegenerate if and only if both $(\Psi|_{[0,\frac{T}{2}]} \R^n,\R^n)$ and $(\Psi|_{[0,\frac{T}{2}]}J_0\R^n,J_0\R^n)$ are nondegenerate.
\end{Prop}
\begin{proof}
We write 
$$
\Psi\big(\tfrac{T}{2}\big)=\left(\begin{array}{cc} A & B\\
C & D  \end{array}\right),\quad 
v=\left(\begin{array}{c} v_1 \\ v_2
 \end{array}\right)\in\R^{2n}
$$
with respect to the decomposition $\R^{2n}=\R^n\oplus J_0\R^n$. Suppose that $\det C=0$. We claim that if $v_1\in\ker C$, $(v_1,0)\in\ker(\Psi(T)-\one_{\R^{2n}})$. Indeed, the claim follows from
$$
I_{\R^{2n}}\Psi\big(\tfrac{T}{2}\big)\left(\begin{array}{c} v_1 \\ 0
 \end{array}\right)=\left(\begin{array}{c} A v_1 \\ 0
 \end{array}\right)= \Psi\big(\tfrac{T}{2}\big) I_{\R^{2n}}\left(\begin{array}{c} v_1 \\ 0
 \end{array}\right)
$$
since $\Psi(T)=I_{\R^{2n}} \Psi(\frac{T}{2})^{-1} I_{\R^{2n}}\Psi(\frac{T}{2})$. The case $\det B=0$ follows in the same manner. 

To show the converse, we assume that $\Psi(T)v=v$ for some $v\in\R^{2n}$. Then since
$$
\left(\begin{array}{cc} Av_1-Bv_2\\
Cv_1-Dv_2  \end{array}\right)=\Psi\big(\tfrac{T}{2}\big) I_{\R^{2n}} \left(\begin{array}{c} v_1\\
v_2  \end{array}\right)=I_{\R^{2n}}\Psi\big(\tfrac{T}{2}\big) \left(\begin{array}{c} v_1\\
v_2  \end{array}\right)=\left(\begin{array}{cc} Av_1+Bv_2\\
-Cv_1-Dv_2  \end{array}\right),
$$
we have $Bv_2=Cv_1=0$. Therefore $\Psi(\frac{T}{2})(v_1,0)=(Av_1,0)$ and $\Psi(\frac{T}{2})(0,v_2)=(0,Dv_2)$.
\end{proof}

The following propositions are also proved in \cite{LZZ06} for their Maslov-type indices and there are analogous statements in higher dimensions.

\begin{Prop}\label{prop:Hormander index}
Suppose that  $\Psi:\R\to\Sp(\R^2)$ is nondegenerate and satisfies $\Psi(0)=\one_{\R^{2n}}$, $\Psi(t+T)=\Psi(t)\Psi(T)$, and $\Psi(-t)=I_{\R^{2n}}\Psi(t) I_{\R^{2n}}$. Then we have
$$
\big|\mu_{\CZ}(\Psi|_{[0,T]})-2\mu_{\RS}(\Psi|_{[0,\frac{T}{2}]}\R,\R)\big|\leq 1.
$$
In consequence, if $\mu_{\CZ}(\Psi|_{[0,T]})\geq3$, 
$$
\mu_{\RS}(\Psi|_{[0,\frac{T}{2}]}\R,\R)\geq\frac{3}{2}.
$$
\end{Prop}
\begin{proof}
The first assertion follows from Proposition \ref{prop:index relation} together with the fact that the H\"ormander index
$$
\mu_H(\Psi|_{[0,\frac{T}{2}]};\R,J_0\R):=\mu_{\RS}(\Psi|_{[0,\frac{T}{2}]}\R,\R)-\mu_{\RS}(\Psi|_{[0,\frac{T}{2}]} J_0\R,J_0\R)
$$
takes values in $\{-1,0,1\}$ which immediately follows from definitions. Then the fact that $\mu_{\RS}(\Psi|_{[0,\frac{T}{2}]}\R,\R)\in\Z+\frac{1}{2}$ under nondegeneracy shows the last inequality.
\end{proof}

\begin{Prop}\label{Prop:RS-index increases} 
Suppose that  $\Psi:\R\to\Sp(\R^2)$ satisfies $\Psi(0)=\one$, $\Psi(t+T)=\Psi(t)\Psi(T)$,  and $\Psi(-t)=I\Psi(t) I$.
If $\mu_{\RS}(\Psi|_{[0,\frac{T}{2}]}\R,\R)\geq\frac{3}{2}$, 
$$
\mu_{\RS}(\Psi|_{[0,\frac{mT}{2}]}\R,\R)\geq \frac{2m+1}{2},\quad m\in\N.
$$
Moreover if $\mu_{\RS}(\Psi|_{[0,\frac{T}{2}]}\R,\R)\geq\frac{1}{2}$, 
$$
\mu_{\RS}(\Psi|_{[0,\frac{mT}{2}]}\R,\R)\geq \frac{1}{2},\quad m\in\N
$$
and if  $\mu_{\RS}(\Psi|_{[0,\frac{T}{2}]}\R,\R)<\frac{1}{2}$, 
$$
\mu_{\RS}(\Psi|_{[0,\frac{mT}{2}]}\R,\R)< \frac{1}{2},\quad m\in\N.
$$
\end{Prop}
\begin{proof}
Let $A_{S_\Psi|_{[0,\frac{T}{2}]}}$ and $A_{S_\Psi|_{[0,\frac{mT}{2}]}}$ be the associated self-adjoint operators with domains $W^{1,2}_I([0,\frac{T}{2}],\R^2)$ and $W^{1,2}_{-I}([0,\frac{mT}{2}],\R^2)$  respectively. Due to the assumption, there exists an eigenfunction $\gamma\in L^{2}_I([0,\frac{T}{2}],\R^2)$ belonging to a negative eigenvalue $\lambda$ of $A_{S_\Psi|_{[0,\frac{T}{2}]}}$ such that 
$$
w(\gamma,\lambda,S_\Psi|_{[0,\frac{T}{2}]})\geq\frac{1}{2}.
$$
We define $\gamma^m$ by concatenating $\gamma$ and $\gamma_I$ $m$-times where $\gamma_I:=I\circ \gamma(\frac{T}{2}-t)$, explicitly,
$$
\gamma^{m}(t):=\left\{\begin{aligned}\gamma\Big(t-T\Big\lfloor\frac{t}{T}\Big\rfloor\Big),\qquad\quad\; &  t\in \mathrm{I}_T,\\[1ex]
\gamma_I\Big(t-T\Big\lfloor\frac{t}{T}\Big\rfloor-\frac{T}{2}\Big), \quad & t\in \Big[0,\frac{mT}{2}\Big]\setminus \mathrm{I}_T \end{aligned}\right.
$$ 
where 
$$
\mathrm{I}_T:= \Big[0,\frac{T}{2}\Big]\cup\Big[T,\frac{3T}{2}\Big]\cup\cdots\cup\Big[ \Big\lfloor \frac{(m-1)}{2}\Big\rfloor T, \Big\lfloor \frac{(m-1)}{2}\Big\rfloor T+\frac{T}{2}\Big].
$$
Then $\gamma^m\in W^{1,2}_I([0,\frac{mT}{2}],\R^2)$ is an eigenfunction of $A_{S_\Psi|_{[0,\frac{mT}{2}]}}$ belonging to the same eigenvalue $\lambda$. The claim follows from 
$$
w\big(\gamma^m,\lambda,S_\Psi|_{[0,\frac{mT}{2}]}\big)=mw\big(\gamma,\lambda,S_\Psi|_{[0,\frac{T}{2}]}\big)\geq \frac{m}{2},\quad m\in\N.
$$
The second claim is proved in the same way. To show the last assertion, we observe that the assumption guarantees that  the eigenvalue $\lambda$ of $A_{S_\Psi|_{[0,\frac{T}{2}]}}$ with $w(\lambda,S_\Psi|_{[0,\frac{T}{2}]})=0$ is positive. $\lambda$ is also a positive eigenvalue of $A_{S_\Psi|_{[0,\frac{mT}{2}]}}$ and $w(\lambda,S_\Psi|_{[0,\frac{mT}{2}]})=0$ for every $m\in\N$. This yields that every negative eigenvalue of $A_{S_\Psi|_{[0,\frac{mT}{2}]}}$ has a negative relative winding number and this proves the last assertion.
\end{proof}

Let $(P,T)$ be a periodic Reeb orbit in a contact manifold $(M,\alpha)$ of dimension $2n+1$, i.e. $P:[0,T]\to M$ with $P(0)=P(T)$, $t\in\R$ and $\dot P(t)=X(P(t))$ where $X$ is the Reeb vector field of $(M,\alpha)$. For a given unitary trivialization 
$$
\Phi_P(t):\R^{2n}\pf\xi_{P(t)},\quad\Phi_P(t)=\Phi_P(t+T),\quad  t\in\R
$$ 
of the contact structure $\xi:=\ker\alpha$ along $P$, we obtain  $\Psi_P:\R\to\Sp(\R^2)$  the linearized Reeb flow along $P$ with respect to the trivialization $\Phi_P$, given by
\beq\label{trivialized linearized flow}
\Psi_P(t):=\Phi_P(t)^{-1}\circ T\phi_X^{t}(P(0))|_{P^*\xi}\circ\Phi_P(0),\quad t\in\R
\eeq
satisfying 
$$
\Psi_P(t+T)=\Psi_P(t)\Psi_P(T).
$$ 
We call a periodic Reeb orbit $(P,T)$ nondegenerate if $\Psi_P|_{[0,T]}$ is nondegenerate.
We define the  \textbf{Conley-Zehnder index} of a periodic Reeb orbit $(P,T)$ with respect to a trivialization $\Phi_P$  by
$$
\mu_{\CZ}^\Phi(P):=\mu_\CZ(\Psi_P|_{[0,T]}).
$$
As observed, $\Psi_P$ defines a loop of symmetric matrices 
\beq\label{loop of symmetric matrices}
S_{P}(t):=-J_0\dot\Psi_P(t)\Psi_P(t)^{-1},\quad S(t)=S(t+T),\quad t\in\R.
\eeq
In dimension 2, we also have
$$
\mu_{\CZ}^\Phi(P)=\mu(S_{P}|_{[0,T]})
$$
if $(P,T)$ is nondegenerate. Note that $S_P$ and $\Psi_P$ depend on the trivialization $\Phi$ along $P$ although the notations do not indicate this.

We consider an additional structure on a contact manifold $(M,\alpha)$, namely a smooth involution $\rho:M\to M$  satisfying 
$$
\rho^*\alpha=-\alpha.
$$ 
We call a triple $(M,\alpha,\rho)$ with such an involution a {\bf real contact manifold}. We shall associate the Robbin-Salamon index to a {\bf symmetric periodic Reeb orbit} $(P,T)$,  i.e. $\im\,P=\rho(\im\,P)$. By time shift, we may assume that 
$$
\rho\circ P(t)=P(-t),\quad t\in\R.
$$
In particular, 
$$
P\Big(\frac{mT}{2}\Big)\in\Fix\rho,\quad m\in\Z.
$$
In order to do this, we need to trivialize $\xi$ symmetrically. For any Hermitian connection $\nabla$ on $\xi$, 
the connection
\beq\label{symmetric connection}
\nabla^\rho_XY:=\frac{1}{2}\bigr(\nabla_XY+T\rho|_\xi\big(\nabla_{(T\rho (X))}{(T\rho|_\xi (Y))}\big)\bigr)
\eeq
is Hermitian again and satisfies
\beq\label{symmetric property}
\nabla^\rho_XY=T\rho|_\xi \big(\nabla^\rho_{(T\rho (X))}{(T\rho|_\xi (Y))}\big).
\eeq
Recall that $\JJ_\rho$ is the space of $d\alpha|_\xi$-compatible  almost complex structures on $\xi$ antiinvariant under $\rho$, i.e. $\rho^*J=-J$ for $J\in\JJ_\rho$.

\begin{Lemma}\label{Lemma:symmetric trivialization}
Let $\varphi:\R\to (M,\alpha,\rho)$ be such that $\varphi(t)=\varphi(t+T)$ and $\rho\circ \varphi(t)=\varphi(-t)$. Then there exists a symmetric unitary trivialization
$$
\wt \Phi: \R\x\R^{2n}\pf \varphi^*\xi,\quad \wt\Phi(t,\cdot)=\wt\Phi(t+T,\cdot)
$$
of $(\xi,J)$ for $J\in\JJ_\rho$ along $\varphi$. To be precise,
$$
\Phi(t)^*d\alpha|_\xi= d\mathbf{x}\wedge d\mathbf{y},\quad J(\varphi(t))\Phi(t)= \Phi(t)J_0,\quad T\rho|_{\xi_{\varphi(t)}}\circ \Phi(t)=\Phi(-t)\circ I_{\R^{2n}}.
$$
where 
$$
\Phi(t):=\wt\Phi(t,\cdot):\{t\}\x\C^n\to \xi_{\varphi(t)},\quad t\in\R.
$$
More generally, let $(D^2\setminus\Gamma,I|_{D^2\setminus\Gamma})$ be a disk with punctures satisfying $I(\Gamma)=\Gamma$ where $\Gamma$ is a finite set of points in $D^2$. For a map $\varphi:(D^2\setminus\Gamma,I|_{D^2\setminus\Gamma})\to (M,\alpha,\rho)$ such that $\rho\circ \varphi=\varphi\circ I|_{D^2\setminus\Gamma}$, there is a symmetric unitary trivialization
$$
(D^2\setminus\Gamma,I|_{D^2\setminus\Gamma})\x(\R^{2n},J_0,I_{\R^{2n}})\pf (\varphi^*\xi,J,T\rho|_{\varphi^*\xi}).
$$
Furthermore if $\Gamma=\emptyset$, it is unique up to homotopy of symmetric unitary trivializations and a multiplication by $\mathrm{diag}(-1,1,\dots, 1)\in O(n)\setminus SO(n)\subset U(n)$. 
\end{Lemma}
\begin{proof}
We fix complex vector space isomorphisms $\Phi(t)$ of $\xi_{\varphi(t)}$ with $\mathbb{C}^n$ such that 
$$
T\rho|_{\xi_{\varphi(t)}}\circ \Phi(t)=\Phi(t)\circ I_{\R^{2n}},\quad t=0,\,\tfrac{T}{2}
$$
 and extend them to whole $\varphi^*\xi$ by using the parallel transport induced by $\nabla^\rho$ given in \eqref{symmetric connection} and by patching trivializations symmetrically. Applying the Gram-Schmidt process, the first two identities follow and the third identity follows from \eqref{symmetric property}. The punctured disk case is also proved in a similar way. We first symmetrically trivialize $\varphi^*\xi$ over $\Fix  I_{D^2\setminus\Gamma}$ and extend a trivialization it to the left hand of $\Fix  I_{D^2\setminus\Gamma}$. Then we extend it to the right hand of $\Fix  I_{D^2\setminus\Gamma}$ in a symmetric way.
 See \cite{MS98} for details.

For the uniqueness statement, we choose  $G:D^2\to U(n)$ with $G(z) I=I G(\bar z)$. Here we use the canonical identification $\R^{2n}=\C^n$. It suffices to consider a half map with boundary condition:
$$
G:D^2\cap\H\to U(n),\quad G(D^2\cap\R)\subset O(n).
$$ 
Then there is a homotopy $\wt G:[0,1]\x (D^2\cap\H)\to U(n)$ such that $\wt G(0,z)=\one_{\R^{2n}}$ or $\mathrm{diag}(-1,1,\dots, 1)\in U(n)$, $\wt G(1,z)=G(z)$, and $\wt G(s,D^2\cap\R)\subset O(n)$. Since a change of two symmetric unitary trivializations is represented by such a map $G$, this proves the last assertion. 
\end{proof}

\begin{Rmk}
To extend the above lemma for surfaces, one needs additional assumptions. A pair $(\Sigma,I_\Sigma)$ is called a {\em  real Riemann surface} if a closed Riemann surface $\Sigma$ carries an anticonformal involution $I_\Sigma:\Sigma\to\Sigma$. In contrast to disks or spheres, not every anticonformal involution on a surface with genus has a fixed locus separating $\Sigma$.  The fixed locus $\Fix I_\Sigma$ of $I_\Sigma$ consists of at most $\mathrm{genus}(\Sigma)+1$ disjoint  Jordan curves, called ovals. If the quotient space $\Sigma/I_\Sigma$, which is called a Klein surface, is orientable, $\Fix I_\Sigma$ divides $\Sigma$ into two connected components. If the Klein surface  $\Sigma/I_\Sigma$ is nonorientable, $\Sigma\setminus\Fix I_\Sigma$ is connected and the number of ovals is at most $\mathrm{genus}(\Sigma)$. In fact, the topological  type of an anticonformal involution is characterized by the number of ovals together with the orientability of the Klein surface. For the study on symmetries on Riemann surfaces, we refer to \cite{BCGG10}. We expect that the above lemma is true for a symmetric map $\varphi$ from $(\Sigma\setminus\Gamma,I_{\Sigma\setminus\Gamma})$ for a nonempty finite set  $\Gamma$ of points in $\Sigma$ if $\Fix I_{\Sigma\setminus\Gamma}$ separates $\Sigma\setminus\Gamma$ and $\Fix(T\rho|_{\varphi^*\xi})$ over $\varphi(\Fix I_{\Sigma\setminus\Gamma})$ is trivial.
\end{Rmk}

For a given symmetric unitary trivialization $\Phi_P$ along a symmetric periodic Reeb orbit $P$ together with the linearized Reeb flow $T\phi_X^t$, we obtain $\Psi_P:\R\to\Sp(\R^{2n})$ and $S_P:\R\to\mathrm{Sym}(\R^{2n})$ given by \eqref{trivialized linearized flow} and \eqref{loop of symmetric matrices} respectively.

\begin{Lemma}\label{lem:I-invariant}
Let $\Psi_P$ and $S_P$ be as above for a given symmetric periodic orbit $(P,T)$. Then,
$$
\Psi_P(-t)=I_{\R^{2n}}\Psi_P(t) I_{\R^{2n}},\quad S_P(-t)=I_{\R^{2n}} S_P(t) I_{\R^{2n}}.
$$
Consequently $S_P(\frac{mT}{2})$, $m\in\Z$ are diagonal matrices.
\end{Lemma}
\begin{proof}
A direct computation shows
\bean
\Psi_P(-t)&=\Phi_P(-t)^{-1}\circ T\phi_X^{-t}(P(0))|_{\xi_{P(0)}}\circ\Phi_P(0)\\
&=\Phi_P(-t)^{-1}\circ T\rho|_{\xi_{P(t)}}\circ T\phi_X^{t}(P(0))|_{\xi_{P(0)}}\circ T\rho|_{\xi_{P(0)}}\circ\Phi_P(0)\\
&=I_{\R^{2n}}\circ\Phi_P(t)^{-1}\circ T\phi_X^{t}(P(0))|_{\xi_{P(0)}}\circ\Phi_P(0)\circ I_{\R^{2n}}\\
&=I_{\R^{2n}}\Psi_P(t) I_{\R^{2n}}.
\eea
The second identity follows from this and shows that $S_P(0)$ is diagonal. Moreover, since $S_P$ is $T$-periodic, 
$$
S_P\big(\tfrac{T}{2}-t\big)=I_{\R^{2n}} S_P\big(-\tfrac{T}{2}+t\big)  I_{\R^{2n}}=I_{\R^{2n}}  S_P\big(\tfrac{T}{2}+t\big)  I_{\R^{2n}}
$$
and hence $S_P(\frac{T}{2})$ is diagonal as well.
\end{proof}

This lemma enables us to associate the Robbin-Salamon index to symmetric periodic orbits. A symmetric periodic orbit $(P,T)$ naturally gives the Reeb chord 
$$
C:=P|_{[0,\frac{T}{2}]}:\Big(\big[0,\tfrac{T}{2}\big],\big\{0,\tfrac{T}{2}\big\}\Big)\to (M,\Fix\rho).
$$
If $\Psi_P$ is the linearized Reeb flow along $P$ with respect to a symmetric trivialization $\Phi$, we define
$$
\Psi_C:=\Psi_P|_{[0,\frac{T}{2}]}:\big[0,\tfrac{T}{2}\big]\to\Sp(\R^{2n}).
$$
Then  the \textbf{Robbin-Salamon index} of a chord $(C,\frac{T}{2})$ with respect to $\Phi$ is defined by
$$
\mu_\RS^\Phi(C):=\mu_\RS(\Psi_C\R^n,\R^n).
$$
A Reeb chord $(C,\frac{T}{2})$ is called nondegenerate if $(\Psi_C\R^n,\R^n)$ is nondegenerate. As observed in Lemma \ref{lem:I-invariant}, we have
$$
D_C:=S_P|_{[0,\frac{T}{2}]}:\Big(\big[0,\tfrac{T}{2}\big],\big\{0,\tfrac{T}{2}\big\}\Big)\to\big(\mathrm{Sym}(\R^{2n}),\mathrm{Diag}(\R^{2n})\big)
$$
and especially in dimension 2,
$$
\mu^\Phi_{\RS}(C)=\mu_I(D_C)
$$
provided that $(C,\frac{T}{2})$ is nondegenerate.\\[-1ex]

In what follows we will describe the indices of symmetric periodic orbits under iteration. For $m\in\N$, we denote by $(P^m,mT)$ the $m$ times iteration of a periodic orbit $(P,T)$, i.e.
$$
P^m:[0,mT]\to M,\quad t\mapsto P\big(t-T\big\lfloor\tfrac{t}{T}\big\rfloor\big).
$$
We can also iterate a chord $(C,\frac{T}{2})$ in $(M,\alpha,\rho)$ with $C(0),\,C(\frac{T}{2})\in\Fix\rho$ in the following sense. We define 
\beqn
C_\rho:\big[0,\tfrac{T}{2}\big]\to M,\quad t\mapsto\rho\circ C\big(\tfrac{T}{2}-t\big).
\eeq
Note that $(C_\rho,\frac{T}{2})$ is also a chord with $C_\rho(0),\,C_\rho(\frac{T}{2})\in\Fix\rho$.
We define $(C^m,\frac{mT}{2})$, $m\in\N$ by
$$
C^{m}(t):=\left\{\begin{aligned}C\Big(t-T\Big\lfloor\frac{t}{T}\Big\rfloor\Big),\qquad\quad\; &  t\in I_T,\\[1ex]
C_\rho\Big(t-T\Big\lfloor\frac{t}{T}\Big\rfloor-\frac{T}{2}\Big), \quad & t\in \Big[0,\frac{mT}{2}\Big]\setminus I_T \end{aligned}\right.
$$
where 
$$
I_T:= \Big[0,\frac{T}{2}\Big]\cup\Big[T,\frac{3T}{2}\Big]\cup\cdots\cup\Big[ \Big\lfloor \frac{(m-1)}{2}\Big\rfloor T, \Big\lfloor \frac{(m-1)}{2}\Big\rfloor T+\frac{T}{2}\Big].
$$
Note that $(C^{2m},mT)=(P^m,mT)$ if $(C,\frac{T}{2})$ is the half-chord of a periodic orbit $(P,T)$. We use a symmetric unitary  trivialization $\Phi_P(t):\R^{2n}\to\xi_{P(t)}$, $t\in\R$ again to trivialize $\xi_{C^m(t)}$ and $\xi_{P^m(t)}$, $m\in\N$. With this choice of trivialization, the following is a direct consequence of Proposition \ref{prop:Hormander index} and Proposition \ref{Prop:RS-index increases}. The assertions concerning the Conley-Zehnder index are well known, see for example \cite{Lon02} or \cite[Appendix]{HWZ03}.

\begin{Cor}\label{cor:index behaviors}
Let $(P,T)$ be a periodic Reeb orbit on $(M,\alpha)$ of dimension 3. For $m\in\N$,
\bean
\mu^\Phi_\CZ(P)< 1&\quad\Longrightarrow\quad\mu^\Phi_\CZ(P^m)< 1,\\
\mu^\Phi_\CZ(P)\geq 1&\quad\Longrightarrow\quad\mu^\Phi_\CZ(P^m)\geq 1,\\
\mu^\Phi_\CZ(P)\geq 3&\quad\Longrightarrow\quad\mu^\Phi_\CZ(P^m)\geq 2m+1.\\
\eea
Let $(C,\frac{T}{2})$ be a Reeb chord on $(M,\alpha,\rho)$ of dimension 3. For $m\in\N$,
\bean
\mu^\Phi_\RS(C)< \frac{1}{2}&\quad\Longrightarrow\quad\mu^\Phi_\RS(C^m)< \frac{1}{2},\\
\mu^\Phi_\RS(C)\geq \frac{1}{2}&\quad\Longrightarrow\quad\mu^\Phi_\RS(C^m)\geq \frac{1}{2},\\
\mu^\Phi_\RS(C)\geq \frac{3}{2}&\quad\Longrightarrow\quad\mu^\Phi_\RS(C^m)\geq  \frac{2m+1}{2}.\\
\eea
If $(P,T)$ is a symmetric periodic orbit with $P=C^2$,
$$
\mu^\Phi_\CZ(P)\geq 3\quad\Longrightarrow\quad\mu^\Phi_\RS(C)\geq \frac{3}{2}.
$$
\end{Cor}

For a contractible periodic orbit, there is a preferred trivialization, namely a trivialization of $\xi$ along a periodic orbit which is extendable to a filling disk. To be precise, we choose any filling disk $\bar P:D^2\to M$ of a periodic orbit $P$, i.e. $\bar P|_{\p D^2}=P$ and use a trivialization of $\bar P^*\xi$ to define the index. Then the Conley-Zehnder index of $P$ does not depend on the choice of trivialization of $\bar P^*\xi$ and we can write $\mu_\CZ(P,\bar P)$. If $\bar P_1,\,\bar P_2$ are two filling disks of $P$, we glue $\bar P_1$ with the opposite orientation and $\bar P_2$ along $P$ and have $\bar P_1\sqcup_P\bar P_2:S^2\to M$. Then a well known fact is that 
$$
\mu_\CZ(P,\bar P_1)=\mu_\CZ(P,\bar P_2)+2c_1(\xi)[\bar P_1\sqcup_P\bar P_2].
$$

A similar identity exists for the Robbin-Salamon index of chords.
By the uniqueness statement in Lemma \ref{Lemma:symmetric trivialization}, a chord $C$ together with a filling 
half-disk $\bar C:(D^2\cap\H,D^2\cap\R)\to (M,\Fix\rho)$, $\bar C|_{\p D^2\cap\H}=C$, uniquely
determines the Robbin-Salamon index, so we write $\mu_\RS(C,\bar C)$. 

\begin{Cor}\label{cor:index+trivialization}
Let $(P,T)$ be a symmetric periodic orbit in $(M,\alpha,\rho)$ and $(C,\frac{T}{2})$ be its half-chord. If $\Phi_1$ and $\Phi_2$ are symmetric unitary trivializations of $\xi_P$, we have the identity
$$
\mu_\RS^{\Phi_1}(C)=\mu_\RS^{\Phi_2}(C)+\mu_\RS(\Phi_1^{-1}\circ\Phi_2|_{[0,\frac{T}{2}]}\R^n,\R^n).
$$
If $\bar P_1$ and $\bar P_2$ are two symmetric filling disks of $P$ so that there are two associated filling half-disks $\bar C_1$ and $\bar C_2$ of $C$,
$$
\mu_\RS(C,\bar C_1)=\mu_\RS(C,\bar C_2)+c_1(\xi)[\bar P_1\sqcup_P\bar P_2].
$$
\end{Cor}
\begin{proof}
Abbreviate
$$
\Psi_i(t):=\Phi_i(t)^{-1}\circ T\phi_X^t(C(t))|_{\xi_C}\circ\Phi_i(0),\quad t\in\Big[0,\frac{T}{2}\Big],\;i\in\{0,1\}.
$$
and 
$$
\Gamma(t):=\Phi_1(t)^{-1}\circ\Phi_2(t),\quad \Gamma_1:=\Gamma|_{[0,\frac{T}{2}]},\quad\Gamma_2:=\Gamma|_{[\frac{T}{2},T]}.
$$
Using Proposition \ref{Prop:loop property of RS-index}, we compute
\bean
\mu_\RS(\Psi_1\R^n,\R^n)&=\mu_\RS(\Phi_1(t)^{-1}\circ\Phi_2(t)\circ\Psi_2(t)\circ\Phi_2(0)^{-1}\circ\Phi_1(0) \R^n,\R^n)\\
&=\mu_\RS(\Phi_1(t)^{-1}\circ\Phi_2(t)\circ\Psi_2(t) \R^n,\R^n)\\
&=\mu_\RS(\Psi_2(t) \R^n,\R^n)+\mu_\RS(\Gamma_1 \R^n,\R^n).
\eea
This proves the first claimed identity. To show the last one, we assume that $\Phi_1$ and $\Phi_2$ are restrictions over $P$ of two symmetric unitary trivializations of $\bar P_1^*\xi$ and $\bar P_2^*\xi$ respectively. Then we have
$$
c_1(\xi)[\bar P_1\sqcup_P\bar P_2]=\frac{1}{2}\mu_\RS(\Gamma V,V)=\mu_\RS(\Gamma_1V,V),
$$
where the last equality uses Proposition \ref{Prop:loop property of RS-index 2}. Hence the first assertion implies the last one.
\end{proof}

For the goals of the present paper, $(M,\alpha)$ will be a nondegenerate, starshaped hypersurface in $\C^2$ invariant under $\tilde\rho:(z_1,z_2)\mapsto(\bar z_1,\bar z_2)$. In particular every periodic orbit is contractible and the first Chern class $c_1(\xi)$ of the contact distribution $\xi=\ker\alpha$ vanishes.  In this case, we have observed that the indices depend neither on the choice of filling (half-) disks nor of (symmetric) unitary trivializations and hence we omit the superscripts indicating the choice of trivializations:
\beq\label{eq:index notation}
\mu_\RS(C)\in\Z+\frac{1}{2},\quad \mu_\CZ(P)\in\Z.
\eeq

\section{Properties of invariant finite energy spheres}
We recall that a symplectic cobordism $(\Xi,\Om)$ is a symplectic manifold with cylindrical ends $\R_-\x M_-$ and  $\R_+\x M_+$ such that there are contact forms $\alpha_\pm$ on $M_\pm$ satisfying 
$$
\Om|_{M_\pm\x\R_\pm}=d(e^{r_\pm}\alpha_\pm).
$$ 
Here $\R_-:=(-\infty,0]$, $\R_+:=[0,\infty)$  and $r_\pm$ are the coordinates on $\R_\pm$. It is often convenient to consider the decomposition 
$$
\Xi=(\R_-\x M_-)\cup_{M_-}\Xi_0\cup_{M_+}(\R_+\x M_+)
$$ 
where $\Xi_0$ is a compact symplectic manifold with boundary $\p\Xi_0=M_-\cup M_+$. A symplectic cobordism may have only a positive/negative cylindrical end or can be compact.  Note that a trivial symplectic cobordism of a contact manifold $(M,\alpha)$ is the symplectization $(\R\x M,d(e^r\alpha))$ of it. An almost complex structure $\wt J$ on $(\Xi,\Om)$ is called {\em compatible} if $\Om(\cdot,\wt J\cdot)$ is a compatible metric, {\em cylindrical} if $\wt J$ preserves contact hyperplanes $\xi_\pm=\ker\alpha_\pm$ and $\wt J\frac{\p}{\p r_\pm}=X_\pm$ on cylindrical ends where $X_\pm$ are the Reeb vector fields of $(M_\pm,\alpha_\pm)$. We denote by $J:=\wt J|_{\xi_\pm}$. From now on we always assume that $\wt J$ is compatible and cylindrical. Let $(S^2=\C\cup\{\infty\},i)$ be a closed Riemann sphere and $\Gamma$ be a finite set of points in $S^2$. We call a map
$$
\tilde u:(S^2\setminus\Gamma,i)\pf(\Xi,\wt J,\Om)
$$
a {\bf finite energy sphere} if it is (pseudo-) holomorphic 
$$
T\tilde u\circ i=\wt J\circ T\tilde u
$$
and has  finite nonzero energy defined by
$$
E(\tilde u):=\sup_{\phi_-\in\CC_-}\int_{\tilde u^{-1}(\R_-\x M_-)}\phi_-\alpha_-+\int_{\tilde u^{-1}(\Xi_0)}\Om+\sup_{\phi_+\in\CC_+}\int_{\tilde u^{-1}(\R_+\x M_+)}\phi_+\alpha_+
$$
where 
$$
\CC_\pm:=\{\phi_\pm\in C^\infty(\R_\pm,[0,1])\,|\,\phi_\pm'\geq 0\}.
$$
In particular, if $\Gamma=\{\infty\}$, we refer to such a map as a {\bf finite energy plane}. If the target $(\Xi,\Om)$ of $\tilde u$ is the symplectization $(\R\x M,d(e^r\alpha))$ of $(M,\alpha)$, we write
$$
\tilde u=(a,u):S^2\setminus\Gamma\pf\R\x M. \\[1ex]
$$

If there is an antisymplectic involution $\tilde\rho$ on a symplectic cobordism $(\Xi,\Om)$ such that $(M_\pm,\alpha_\pm,\rho:=\tilde\rho|_{M_\pm})$ are real contact manifolds, a triple $(\Xi,\Om,\tilde \rho)$ is called a {\bf real symplectic cobordism}. A symplectization $(\R\x M,d(e^r\alpha),\tilde\rho)$ of a real contact manifold $(M,\alpha,\rho)$ where $\tilde\rho(r,m):=(r,\rho(m))$ for $(r,m)\in\R\x M$, is an example. Suppose that a compatible cylindrical almost complex structure $\wt J$ is {\bf antiinvariant} under $\tilde\rho$, i.e. 
$$
\tilde\rho^*\wt J=-\wt J.
$$
The involution $I$ on $\C$ defines an involution on $S^2=\C\cup\{\infty\}$. Suppose that the set of punctures $\Gamma$ is symmetric, i.e. $I(\Gamma)=\Gamma$. Then $I$ defines an involution on $S^2\setminus\Gamma$ and we denote it also by $I$. Then if $\tilde u:(S^2\setminus\Gamma,i)\to(\Xi,\wt J)$ is holomorphic, so is $\tilde\rho\circ \tilde u\circ I$. We denote by $\JJ_{\tilde \rho}$ the space of antiinvariant compatible cylindrical almost complex structures. We also denote by $\JJ_{\rho}$ the space of compatible almost complex structure on $\xi_\pm$ which is antiinvariant under $\rho$. Then $J\in\JJ_{\rho}$ provided $\wt J\in\JJ_{\tilde\rho}$. Henceforth an almost complex structure $\wt J$ is tacitly assumed to be an element in $\JJ_{\tilde\rho}$. A finite energy sphere $\tilde u:(S^2\setminus\Gamma,I)\to(\Xi,\tilde\rho)$ is called {\bf invariant} if
$$
\tilde u=\tilde\rho\circ \tilde u\circ I.
$$
or equivalently if $\tilde u(z)=\tilde\rho\circ\tilde u(\bar z)$, $z\in S^2$. If this is the case, the image of an invariant finite energy sphere $\tilde u$ is invariant under the involution $\tilde\rho$ and $\tilde u(\Fix I)\subset\Fix\tilde\rho$. The fixed locus $\Fix I$ divides $S^2$ into two hemispheres and we denote by $D_+$ be the upper hemisphere with boundary. We denote by $\Gamma_I$ the set of punctures on $D_+$, by $\Gamma_o$ the set of interior punctures on $D_+$, and by  $\Gamma_\p$ the set of boundary punctures on $D_+$, i.e. 
$$
\Gamma_\p:=\Gamma\cap\Fix I,\quad\Gamma_o:=\Gamma\cap\mathring D_+,\quad \Gamma_I:=D_+\cap\Gamma=\Gamma_\p\sqcup\Gamma_o.
$$ 
We denote a holomorphic map $u_I$ from $(D_+\setminus\Gamma_I,i)$ to $(\Xi,\wt J)$ satisfying the boundary condition $\tilde u_I(\Fix I\setminus\Gamma_\p)\subset\Fix\tilde\rho$ by
$$
\tilde u_I:(D_+,\Fix I)\setminus\Gamma_I\to(\Xi,\Fix\tilde\rho).
$$
and call it a {\bf finite energy half-sphere} if $E(\tilde u_I)<\infty$. If $\Xi=\R\x M$, we write 
$$
\tilde u_I=(a_I,u_I):(D_+,\Fix I)\setminus\Gamma_I\to(\R\x M,\Fix\tilde\rho).
$$
If $\Gamma=\{\infty\}$, $D_+\setminus\Gamma$ coincides with the upper half-plane $\H:=\{z\in\C\,|\,\mathrm{Im}(z)\geq0\}$ and $\tilde u_I$ is called a {\bf finite energy half-plane}.

\subsection{Asymptotic behavior}

Let $\tilde u=(a,u):S^2\setminus\Gamma\to\Xi$ be a finite energy sphere. A puncture $z\in\Gamma$ with a small neighborhood mapped to a bounded region can be removed (i.e. $\tilde u$ extends smoothly over $z$) due to removal of singularities. Let $z\in\Gamma$ be a nonremovable puncture and $\UU(z)$ be a sufficiently small neighborhood of $z$. In view of \cite{Hof93}, $\tilde u$ maps $\UU(z)$ into the cylindrical ends $\R_\pm\x M_\pm$. Moreover, if we write 
$$
\tilde u|_{\UU(z)}=(a,u):\UU(z)\to\R_\pm\x M_\pm,
$$
$a(z')\to\pm\infty$ as $z'\to z$ and $\tilde u$ asymptotically converges (not necessarily uniformly) to a periodic orbit of the Reeb vector field $X_\pm$ of $(M_\pm,\alpha_\pm)$ at $z\in\Gamma$ respectively. From now one we tacitly assume that $\pi\circ Tu\neq 0$ on $\UU(z)$, where $\pi:TM_\pm\to\xi_\pm:=\ker\alpha_\pm$, so that $\tilde u|_{\UU(z)}$ is not a trivial cylinder over an asymptotic periodic orbit. We denote by $\Gamma_\pm\subset\Gamma$ the set of positive/negative punctures approaching periodic orbits of $X_\pm$ respectively. We assume that there is no removable puncture, i.e. 
$$
\Gamma=\Gamma_-\sqcup\Gamma_+.
$$ 
Note that a finite energy plane inside a symplectization always has a unique positive puncture due to the maximum principle. If the asymptotic periodic orbit is nondegenerate, a further study on the asymptotic behavior is carried out by Hofer, Wysocki, and Zehnder \cite{HWZ96}. We recall some of their results indispensable to our story. We choose holomorphic polar coordinates $\phi_z:\R_+\x S^1\to \UU(z)\setminus z$ in a sufficiently small neighborhood $\UU(z)$ of $z\in\Gamma$. We denote by $\Gamma^2(\xi_P)$ resp. $\Gamma^{1,2}(\xi_P)$ the space of $L^2$- resp. $W^{1,2}$-sections of $\xi_P\to P$ for a periodic orbit $(P,T)$ of $(M_\pm,X_\pm)$.

\begin{Thm}[\!\!\cite{HWZ96}]\label{thm:asymptotic formula}
Let $\tilde u:S^2\setminus\Gamma\to \Xi$ be a finite energy sphere converging to a nondegenerate periodic orbit $(P,T)$ of $(M_+,X_+)$ at $z\in\Gamma^+$ asymptotically. Then 
$$
\lim_{s\to\infty}u\circ\phi_z(s,t)=P(Tt)
$$
in $C^\infty(S^1,M)$. Moreover, $\tilde u$ has the following asymptotic formula near the asymptotic orbit. 
\bean\label{eq:asymptotic representation}
u\circ\phi_z(s,t)=\exp_{P(Tt)}[e^{\lambda s}(e(t)+r(s,t))]
\eea
where 
\begin{itemize}
\item[(i)] $r(s,t)\in\xi_{P(Tt)}$ converges to $0$ uniformly in $t\in S^1$ with all derivatives as $s\to\infty$;
\item[(ii)] $\lambda\in\R$ is a negative eigenvalue of the asymptotic operator $A_P$ on $\Gamma^2(\xi_P)$ with $\dom A_P=\Gamma^{1,2}(\xi_P)$ by
$$
A_P(v):=-J(\nabla_tv-T\nabla_vX_+);
$$
\item[(iii)] $e(t)\in\Gamma^{1,2}(\xi_P)$ is an eigenfunction of $A_P$ belonging to $\lambda$.
\end{itemize}
For a negative puncture $z\in\Gamma^-$ where $\tilde u$ converges to a nondegenerate periodic orbit  $(P,T)$ of $(M_-,X_-)$, the same statements remain true with $P(Tt)$ replaced by $P(-Tt)$ and with a negative eigenvalue $\lambda\in\sigma(A_P)$ replaced by a positive eigenvalue.
\end{Thm}

Let  $\Gamma$ be symmetric, i.e. $I(\Gamma)=\Gamma$ and let holomorphic polar coordinates $\phi_z:\R_+\x S^1\to\UU(z)\setminus z$ at $z\in\Gamma_\p$ satisfy
\beq\label{eq:hol polar coord}
\phi_z(s,-t)=I\circ\phi_z(s,t).
\eeq
Suppose that $\tilde u:(S^2\setminus\Gamma,I)\to(\Xi,\tilde\rho)$ is an invariant finite energy sphere with respect to $\wt J\in\JJ_{\tilde \rho}$. Note that a nonremovable puncture $z\in\Gamma_\p$ converges to a symmetric periodic orbit. In this situation,
$$
 u\circ\phi_z(s,0),\; u\circ\phi_z\Big(s,\frac{1}{2}\Big)\in\Fix\rho.
$$ 
Since $r(s,t)\to0$ as $s\to\infty$, $T\rho(e(t))=e(-t)$ and $r(s,0),\,r(s,\frac{1}{2})\in\Fix \rho$. Hence,
$$
e(t)\in\Fix T\rho|_{\xi_{P(t)}},\quad T\rho|_{\xi_{P(t)}}\dot e(t)=-\dot e(t),\quad t=0,\,\frac{1}{2}.
$$
Suppose that $(P,T)$ is a symmetric periodic orbit. Then $C:=P|_{[0,\frac{T}{2}]}$ is a Reeb chord satisfying the boundary condition $C(0),\,C(\frac{T}{2})\in\Fix\rho$. 
We denote by  $\Gamma_\rho^{1,2}(\xi_C)$ the space of $W^{1,2}$-sections $v(t)$ with boundary conditions $v(t)\in\Fix T\rho|_{\xi_{C(t)}}$ for $t=0,\,\frac{1}{2}$. Then the asymptotic operator $A_P$ in the theorem descends to an operator $A_C$ on $\Gamma^2(\xi_C)$ with $\dom A_C=\Gamma^{1,2}_\rho(\xi_C)$ since $J\in\JJ_\rho$ and $\rho^*X_\pm=-X_\pm$. Hence the following corollary directly follows from the theorem. 

\begin{Cor}\label{Cor:asymptotic formula for invariant finite energy planes}
Let $\tilde u:(S^2\setminus\Gamma,I)\to(\Xi,\tilde\rho)$ be an invariant finite energy sphere converging to a nondegenerate symmetric periodic orbit $(P,T)$ of $(E_+,X_+)$ at $z\in\Gamma_\p^+$ asymptotically and let $\tilde u_I:(D_+,\Fix I)\setminus\Gamma_I\to(\Xi,\Fix \tilde\rho)$ be the associated finite energy half-sphere.
Then 
$$
\lim_{s\to\infty}u_I\circ\phi_z(s,t)=C(Tt)
$$
in $C^\infty([0,\frac{1}{2}],M)$ where $(C,\frac{T}{2})$ is the half-chord of $(P,T)$. Moreover, $\tilde u_I$ has the following asymptotic formula near the asymptotic orbit: 
$$
u_I\circ\phi_z(s,t)=\exp_{C(Tt)}[e^{\lambda_\rho s}(e_\rho(t)+r_\rho(s,t))]
$$
where 
\begin{itemize}
\item[(i)] $r_\rho(s,t)\in\xi_{C(Tt)}$ converges to $0$ uniformly in $t\in[0,\frac{1}{2}]$ with all derivatives as $s\to\infty$;
\item[(ii)]  $\lambda_\rho\in\R$ is a negative eigenvalue of the asymptotic operator $A_C$ on $\Gamma^2(\xi_C)$;
\item[(iii)] $e_\rho(t)\in\Gamma_\rho^{1,2}(\xi_C)$ is an eigenfunction of $A_{C}$ belonging to $\lambda_\rho$.
\end{itemize}
For a negative puncture $z\in\Gamma_\p^-$ where $\tilde u$ converges to a nondegenerate symmetric periodic orbit  $(P,T)$ of $(M_-,X_-)$, the same statements remain true with $C(Tt)$ replaced by $C(-Tt)$  for the half-chord $C$ of $P$ and with a negative eigenvalue $\lambda_\rho\in\sigma(A_C)$ replaced by a positive eigenvalue.
\end{Cor}

\begin{Rmk}
The asymptotic behavior for general finite energy half-planes with totally real boundary conditions is studied by Abbas \cite{Abb04}.
\end{Rmk}

\subsection{Winding numbers}

We briefly recall a notion of winding numbers associated to finite energy spheres from \cite{HWZ95b} and extend this to finite energy half-spheres. Let $\tilde u=(a,u):S^2\setminus\Gamma\to\R\x M$ be a finite energy sphere with $\Gamma\neq\emptyset$ in the symplectization $(\R\x M,d(e^r\alpha))$ of $(M,\alpha)$. Due to  holomorphicity, $\pi\circ Tu(z)$, $z\in S^2\setminus\Gamma$ where $\pi:TM\to\xi$ is complex linear:
$$
\pi\circ Tu(z)\in\Hom_\C\big(T_z (S^2\setminus\Gamma),\xi_{u(z)}\big).
$$
We keep assuming that $\pi\circ Tu\neq0$ to exclude the trivial case that $\tilde u$ is a trivial cylinder over a periodic orbit. Let $\Phi$ be a unitary trivialization of $u^*\xi$ discussed in the previous section. Using $\Phi$ and the trivialization $\UU(z)\cong\R_+\x S^1$ given by $\phi_z$ at $z\in\Gamma^\pm$, we trivialize the complex line bundle $\Hom_\C(T(S^2\setminus\Gamma),u^*\xi)$ over $S^2\setminus\Gamma$ and also the section $\pi\circ Tu$ of it. We denote the trivialization of the section $\pi\circ Tu$ with respect to $\Phi$ near $z\in\Gamma$ by
$$
\gamma_z^\Phi: \UU(z)\to\Hom_\C(\C,\C).
$$ 
Then we abbreviate
\beq\label{eq:rot map}
\rho^\Phi_z(s):S^1\to\C,\quad t\mapsto \big(\gamma_z^\Phi\circ\phi_z(s,\pm t)\big)(e^{\pm it}),\quad z\in\Gamma^\pm
\eeq
respectively and the {\bf winding number} at $z\in\Gamma$ is defined by
$$
\wind^\Phi_\infty(\tilde u;z):=\lim_{s\to\infty}\frac{1}{2\pi} \int_{S^1} (\rho^\Phi_z(s))^* d\theta
$$
where $\theta$ is the angular coordinate on $\C$. The winding number at each puncture depends on the choice of the trivialization $\Phi$ of $u^*\xi$ but the total winding number
$$
\wind_\infty(\tilde u):=\sum_{z\in\Gamma^+}\wind^\Phi_\infty(\tilde u;z)-\sum_{z\in\Gamma^-}\wind^\Phi_\infty(\tilde u;z)
$$
is independent of this choice, see \cite{HWZ95b}.

\begin{Prop}\cite[Proposition 4.1]{HWZ95b}
Let $\tilde u=(a,u):S^2\setminus\Gamma\to\R\x M$ be a finite energy sphere with nondegenerate asymptotic periodic orbits. Then 
$$
\ZZ(\pi\circ Tu):=\{z\in S^2\setminus\Gamma\,|\,\pi\circ Tu(z)=0\}
$$
is a finite set and each zero $z\in\ZZ(\pi\circ Tu)$ has a positive degree.
\end{Prop}

We denote by $\wind_\pi(\tilde u)$ the sum of the degrees of zeros in $\ZZ(\pi\circ Tu)$. As a consequence of the above proposition, it holds that 
$$
0\leq \wind_\pi(\tilde u)<\infty.
$$

\begin{Prop}\label{Prop:identity betn wind numbers} {\cite[Proposition 5.6]{HWZ95b}}
Let $\tilde u=(a,u):S^2\setminus\Gamma\to\R\x M$ be a finite energy sphere with nondegenerate asymptotic periodic orbits. Then 
$$
\wind_\infty(\tilde u)=\wind_\pi(\tilde u)+(2-\#\Gamma).
$$
\end{Prop}
\quad\\[-2ex]

In what follows, we shall generalize the notion of such winding numbers to finite energy half-spheres. For an invariant finite energy sphere $\tilde u=(a,u):(S^2,I)\to(\R\x M,\tilde\rho)$ for $\wt J\in\JJ_{\tilde \rho}$  with nondegenerate asymptotic periodic orbits, consider a finite energy half-sphere $\tilde u_I=(a_I,u_I):(D_+,\Fix I)\setminus\Gamma_I\to(\R\x M,\Fix\tilde\rho)$. We assume that $\Gamma_\p$ is not empty. We have
$$
\pi\circ Tu_I(z)\in\Hom_\C\big((T_z D_+,T_z\Fix I),(\xi_{u(z)},T_{u(z)}\Fix\rho)\big)
$$
which means that 
$$
\pi\circ Tu_I(z)\in\Hom_\C(T_zD_+,\xi_{u(z)})
$$
and for $z\in\Fix I\setminus\Gamma_\p$, it additionally satisfies that
$$
\pi\circ Tu_I(z)|_{T_z\Fix I}\in\Hom_\R(T_z\Fix I,T_{u(z)}\Fix\rho).
$$
We abbreviate
$$
\ZZ(\pi\circ Tu_I):=\{z\in D_+\setminus\Gamma_I\,|\,\pi\circ Tu_I(z)=0\}.
$$

Using a symmetric unitary trivialization $\Phi$ of $u^*\xi$ as in Lemma \ref{Lemma:symmetric trivialization} and symmetric holomorphic polar coordinates $\phi_z:\R_+\x S^1\to \UU(z)$ in \eqref{eq:hol polar coord} on a small neighborhood $\UU(z)$ of  $z\in\Gamma_\p$, the section $\pi\circ Tu$ is written  as follows.
$$
\gamma_z^\Phi:\big(\UU(z),\Fix I|_{\UU(z)}\big)\to\Hom_\C\big((\C,\R),(\C,\R)\big)
$$
The map defined in \eqref{eq:rot map} satisfies
$$
\rho^\Phi_z(s):\big(S^1,\big\{0,\tfrac{1}{2}\big\}\big)\to(\C,\R).
$$
The {\bf relative winding number} at $z\in\Gamma_\p$ is defined by
$$
\wind^\Phi_\infty(\tilde u_I,z):=\lim_{s\to\infty}\frac{1}{2\pi} \int_0^{\frac{1}{2}}(\rho^\Phi_z(s))^* d\theta\in\frac{1}{2}\Z
$$
where $\theta$ is the angular coordinate on $\C$. The winding number of $\tilde u_I$ at $z\in\Gamma_o$ is defined as before, i.e. 
$$
\wind^\Phi_\infty(\tilde u_I;z):=\wind^\Phi_\infty(\tilde u,z).
$$
As before, each (relative) winding number depends on the symmetric trivialization $\Phi$ of $u_I^*\xi$ whereas the total winding number 
$$
\wind_\infty(\tilde u_I):=\sum_{z\in\Gamma_I^+}\wind^\Phi_\infty(\tilde u_I;z)-\sum_{z\in\Gamma_I^-}\wind^\Phi_\infty(\tilde u_I;z)
$$
does not. 

We define $\wind_\pi(\tilde u_I):=\frac{\wind_\pi(\tilde u)}{2}$. It is easy to show that this agrees with the sum of the (half-) degrees of zeros in $\ZZ(\pi\circ Tu_I)=\ZZ(\pi\circ Tu)\cap D_+$. 

Recall that the Robbin-Salamon index of a chord $(C,\frac{T}{2})$ cannot fully determine the Conley-Zehnder index of a symmetric periodic orbit $(P=C^2,T)$. Nonetheless, the following proposition shows that the winding number of $\tilde u_I$ is exactly half of the winding number of $\tilde u$. This simple observation turns out to be crucial in the proof of our main result.

\begin{Prop}\label{Prop:wind=2 relative wind}
Let $\tilde u=(a,u):(S^2\setminus\Gamma,I)\to(\R\x M,\wt J,\tilde\rho)$ be  an invariant finite energy sphere as above. Then,
$$
 \wind_\infty(\tilde u)=2\wind_\infty(\tilde u_I)
$$
\end{Prop}
\begin{proof}
Since $T\rho\circ Tu=Tu\circ TI$ and all involved maps respect symmetry, we have
$$
\rho_z^\Phi(s,-t)=I\circ\rho_z^\Phi(s,t)
$$
for $(s,t)\in\R_+\x S^1$ and $z\in\Gamma_\p$ and therefore $2\wind^\Phi_\infty(\tilde u_I;z)=\wind^\Phi_\infty(\tilde u;z)$. Moreover, since punctures in $\Gamma\setminus\Fix I$ appear in pairs with the same winding number, the claim is proved.
\end{proof}

The following corollary is a direct consequence of Proposition \ref{Prop:identity betn wind numbers} and Proposition \ref{Prop:wind=2 relative wind}.
\begin{Cor}\label{cor:identity for winding numbers}
For $\tilde u$ in Proposition \ref{Prop:wind=2 relative wind}, we have
$$
\wind_\infty(\tilde u_I)=\wind_\pi(\tilde u_I)+\frac{1}{2}(2-\#\Gamma).
$$
\end{Cor}

A well known fact is that parity of $\mu^\Phi_{\CZ}(P)$ for a periodic orbit $(P,T)$ remains unchanged under a change of trivialization of $\Phi$ of $\xi_P$. So we can denote by 
$$
p(P)\in\{0,1\}
$$ 
the parity of $\mu^\Phi_{\CZ}(P)$ for any trivialization $\Phi$, i.e. $p(P)=1$ if $\mu^\Phi_{\CZ}(P)$ is odd and $p(P)=0$ otherwise, and $\Gamma=\Gamma_\mathrm{odd}(\tilde u)\sqcup\Gamma_\mathrm{even}(\tilde u)$ where $\Gamma_\mathrm{odd}(\tilde u)$ resp. $\Gamma_\mathrm{even}(\tilde u)$ is the set of punctures with odd resp. even indices. Like the winding number the total index of a finite energy sphere $\tilde u:S^2\setminus\Gamma\to\R\x M$
$$
\mu(\tilde u):=\sum_{z\in\Gamma^+}\mu_\CZ^\Phi(P_z)-\sum_{z\in\Gamma^-}\mu_\CZ^\Phi(P_z)
$$
where $P_z$ is an asymptotic periodic orbit at $z\in\Gamma$, does not depend on the choice of the trivialization $\Phi$ of $u^*\xi$. 
 
Using the asymptotic representation of $\tilde u$, we are able to compare the winding number of $\tilde u$ at $z\in\Gamma$ with the index of the corresponding asymptotic periodic orbit $(P,T)$.

\begin{Prop}\cite{HWZ95b} \label{Prop:wind and index}
Let $\tilde u=(a,u):S^2\setminus\Gamma\to\R\x M$ be a finite energy sphere which converges to a nondegenerate periodic  orbit $(P,T)$ at $z\in\Gamma$. It holds that for any unitary trivialization $\Phi$ of $u^*\xi$,
$$
\wind_\infty^\Phi(\tilde u;z)\leq\frac{1}{2}\big(\mu^\Phi_{\CZ}(P)-p(P)\big),\quad z\in\Gamma^+
$$
and 
$$
\wind_\infty^\Phi(\tilde u;z)\geq\frac{1}{2}(\mu^\Phi_{\CZ}(P)+p(P)),\quad z\in\Gamma^-.
$$
In consequence, we have
$$
\mu(\tilde u)\geq 2\wind_\infty(\tilde u)+\#\Gamma_\mathrm{odd}(\tilde u)
$$
If in particular $\tilde u$ is a finite energy plane, $\mu_\CZ(P)\geq 2$.
\end{Prop}

Imitating the above proposition, we prove the corresponding result for finite energy half-spheres. The total index of a finite energy half-sphere $\tilde u_I$ is 
$$
\mu(\tilde u_I):=\sum_{z\in\Gamma_\p^+}\mu^\Phi_{\RS}(C_z)+\sum_{z\in\Gamma_o^+}\mu_\CZ^\Phi(P_z)-\sum_{z\in\Gamma_\p^-}\mu_\RS^\Phi(C_z)
-\sum_{z\in\Gamma_\p^-}\mu_\CZ^\Phi(P_z)
$$
where $P_z$ and $C_z$ are asymptotic periodic orbits and chords at $z\in\Gamma_I=\Gamma_o\sqcup\Gamma_\p$ respectively, and again this is independent of the choice of symmetric trivializations $\Phi$ of $u^*\xi$.

\begin{Prop}\label{Prop:relative wind and index}
Let $\tilde u=(a,u):S^2\setminus\Gamma\to\R\x M$ be an invariant finite energy sphere. If $\tilde u$ converges to a nondegenerate symmetric periodic orbit $(P,T)$ at $z\in\Gamma_\p$,  we have 
$$
\wind^\Phi_\infty(\tilde u_I;z)\leq\frac{1}{2}\Big(\mu^\Phi_{\RS}(C)-\frac{1}{2}\Big),\quad z\in\Gamma^+_\p
$$
and 
$$
\wind^\Phi_\infty(\tilde u_I;z)\geq\frac{1}{2}\Big(\mu^\Phi_{\RS}(C)+\frac{1}{2}\Big),\quad z\in\Gamma^-_\p
$$
where $(C,\frac{T}{2})$ is the half-chord of $(P,T)$. In consequence, we have 
$$
\mu(\tilde u_I)\geq 2\,\wind_\infty(\tilde u_I)+\frac{\#\Gamma_\p}{2}+\#(\Gamma_{o}\cap\Gamma_\mathrm{odd}).
$$
If in particular $\tilde u$ is an invariant finite energy plane, $\mu_\RS(C)\geq \frac{3}{2}$.
\end{Prop}
\begin{proof}
Let $e_\rho(t)\in\Gamma^{1,2}(\xi_C)$ be an eigenfunction of $A_C$ representing the asymptotic convergence rate of $u_I$, see Corollary \ref{Cor:asymptotic formula for invariant finite energy planes}. Restricting the trivialization $\Phi$ to $C$, we have $\Phi_C(t):\C\to\xi_{C(Tt)}$. If we write $D_\Phi(t)=-J_0\dot\Psi(t)\Psi(t)^{-1}$ where  $\Psi(t)=\Phi_C(t)^{-1}\circ T\phi_{X_\pm}^{Tt}(C(0))|_{\xi_{C}}\circ\Phi_C(t)$, the asymptotic operator $A_C$ is written as 
$$
A_{D_\Phi}=-J_0\frac{\p}{\p t}-D_\Phi(t):W^{1,2}_I([0,\tfrac{1}{2}],\R^2)\subset L^2([0,\tfrac{1}{2}],\R^2)\to L^2([0,\tfrac{1}{2}],\R^2)
$$
with respect to the trivialization  $\Phi_C$. Then 
$$
e^\Phi_\rho\in L^2([0,\tfrac{1}{2}],\R^2),\quad  t\mapsto\Phi_C(t)^{-1}(e_\rho(t))
$$ 
 is an eigenfunction of $A_{D_\Phi}$ belonging to a certain negative/positive eigenvalue $\lambda_\mp\in\R$.
Therefore we have
$$
\wind^\Phi_\infty(\tilde u_I;z)=w(e^\Phi_\rho,\lambda_\mp,D_\Phi),\quad z\in\Gamma_\p^\pm
$$
which in turn implies
$$
\wind^\Phi_\infty(\tilde u_I;z)\leq\alpha_I(D_\Phi),\quad z\in\Gamma^+_\p
$$
and 
$$
\wind^\Phi_\infty(\tilde u_I;z)-\frac{1}{2}\geq\alpha_I(D_\Phi),\quad z\in\Gamma^-_\p.
$$
Since $\mu^\Phi_{\RS}(C)=2\alpha_I(D_\Phi)+\frac{1}{2}$, the first two inequalities are proved. These together with Proposition \ref{Prop:wind and index} show the third inequality. The last assertion concerning an invariant finite energy plane follows from the inequality $\wind_\pi(\tilde u_I)\geq 0$ and Corollary \ref{cor:identity for winding numbers}.
\end{proof}

\subsection{Transversality}

We have chosen antiinvariant almost complex structures $\JJ_{\tilde \rho}$ to consider invariant finite energy spheres. Since this choice is  restrictive, one cannot achieve transversality and this obstructs the study of moduli spaces of finite energy spheres in general. Our idea to get round this difficulty is to use the facts that somewhere injective finite energy half-spheres are regular for a generic $\wt J\in\JJ_{\tilde\rho}$ and that finite energy planes in a symplectization always have index large enough to satisfy automatic transversality.

\begin{Prop} \label{prop:automatic transversality}\cite{HWZ99}
Any finite energy plane $\tilde u:\C\to\R\x M$ with $\mu(\tilde u)\in\{2,3\}$ is regular for every compatible cylindrical $\wt J$.
\end{Prop}
For a further study on automatic transversality for finite energy spheres, we refer the reader to Wendl's work \cite{Wen10}.

Recall that a finite energy sphere $\tilde u:S^2\setminus\Gamma\to\Xi$ is called somewhere injective if there is a so called injective point $z\in S^2\setminus\Gamma$ such that $d\tilde u(z)\neq0$ and $\tilde u^{-1}(\tilde u(z))=\{z\}$.

\begin{Def}
A finite energy half-sphere $\tilde u_I:(D_+,\Fix I)\setminus\Gamma_I\to(\Xi,\Fix\tilde\rho)$ is called {\em somewhere injective} if its double $\tilde u: S^2\setminus\Gamma\to\Xi$ is somewhere injective. A point $z\in D_+\setminus\Gamma_I$ is called an {\em injective point} of $\tilde u_I$ if 
$$
d\tilde u_I(z)\neq0,\quad \tilde u_I^{-1}(\tilde u_I(z))=\{z\},\quad 
\tilde u_I^{-1}(\tilde\rho\circ\tilde u_I( z))=\left\{\begin{array}{ll} \{z\} \quad & z\in\Fix I, \\[1ex] \;\;\emptyset & z\notin\Fix I.\end{array}\right.
$$
\end{Def}

Note that if $\tilde u_I$ is somewhere injective, $\im\,\tilde u_I\neq \im\,(\tilde\rho\circ\tilde u_I)$. 

\begin{Lemma}
If a finite energy half-sphere $\tilde u_I:(D_+,\Fix I)\setminus\Gamma_I\to(\Xi,\Fix\tilde\rho)$ is somewhere injective, the set of injective points is open and dense.
\end{Lemma}
\begin{proof}
We recall that the set $\II(\tilde u)$ of injective points of a somewhere injective curve $\tilde u:S^2\setminus\Gamma\to\Xi$ is open and dense, see \cite{HWZ95b}. We claim that the open dense subset $\II(\tilde u)|_{D_+}$ of $D_+$ consists of injective points of $\tilde u_I$. The first two properties are obvious. The last requirement follows from the observation that for $z\in \II(\tilde u)$, $\tilde u^{-1}(\tilde\rho\circ\tilde u( z))=I(\tilde u^{-1}(\tilde u(z)))=\{I(z)\}$.
\end{proof}

Now we prove that a somewhere injective finite energy half-sphere is simple.

\begin{Thm}\label{invariant branched covering}
If an invariant finite energy sphere $\tilde u:(S^2\setminus\Gamma,I)\to(\Xi,\tilde\rho)$ is not somewhere injective, there exist a set of punctures $\underline\Gamma\subset S^2$ with $I(\underline\Gamma)=\underline\Gamma$, a holomorphic map $p:S^2\setminus\Gamma\to S^2\setminus\underline\Gamma$ with  $\deg(p)>1$, and a somewhere injective invariant finite energy sphere 
$$
\underline{\tilde u}:(S^2\setminus\underline\Gamma,I)\to(\Xi,\tilde\rho)
$$ 
satisfying  
$$
p\circ I=I\circ p,\quad\tilde u=\underline{\tilde u}\circ p. 
$$
Furthermore the map $p$ is a complex polynomial with real coefficients.
\end{Thm}
\begin{proof}
Due to \cite[Section 6]{HWZ95b}, there exist an underlying finite energy sphere $\underline{\tilde u}: S^2\setminus\underline\Gamma,\to(\Xi,\wt J)$ and a complex polynomial $p:S^2\setminus\Gamma\to S^2\setminus\underline\Gamma$ of $\deg(p)>1$ such that $\tilde u=\underline{\tilde u}\circ p$.  Let $Z(\underline{\tilde u})$ be the set of noninjective points of $\underline{\tilde u}$, i.e. 
$$
Z(\underline{\tilde u})=\big\{z\in S^2\setminus\underline\Gamma\,\big|\,d\underline{\tilde u}(z)=0 \;\textrm{ or }\; \underline{\tilde u}^{-1}(\underline{\tilde u}(z))\neq\{z\}\big\}.
$$
According to \cite{MS04} together with the asymptotic formula, $Z(\underline{\tilde u})$ is countable and can only accumulate at critical points of $\underline{\tilde u}$ whose cardinality is finite. Since $w:=\underline{\tilde u}|_{S^2\setminus(\underline\Gamma\cup Z(\underline{\tilde u}))}$ is an embedding and its image is invariant under $\tilde\rho$, we consider the involution $w^{-1}\circ\tilde\rho\circ w$ on $S^2\setminus (\underline\Gamma\cup Z(\underline{\tilde u}))$. Using $\tilde\rho^*\wt J=-\wt J$ and $w^*\wt J=i$, we deduce
$$
(w^{-1}\circ\tilde\rho\circ w)^*i=(\tilde\rho\circ w)^*\wt J=-w^*\wt J=-i.
$$
In other words, $w^{-1}\circ\tilde\rho\circ w$ is an anticonformal involution on $S^2\setminus (\underline\Gamma\cup Z(\underline{\tilde u}))$. We think of that 
$$
w^{-1}\circ\tilde\rho\circ w:S^2\setminus (\underline\Gamma\cup Z(\underline{\tilde u}))\to S^2\setminus \underline\Gamma.
$$ 
Due to the asymptotic behavior of $\underline{\tilde u}$, the image of an open neighborhood of $z\in Z(\underline{\tilde u})$ under $w^{-1}\circ\tilde\rho\circ w$ is bounded inside $\Xi$. Therefore points in $Z(\underline{\tilde u})$ are removable singularities and we obtain the extended map 
$$
\underline I:S^2\setminus \underline\Gamma\to S^2\setminus \underline\Gamma,\qquad \underline I|_{S^2\setminus (\underline\Gamma\cup Z(\underline{\tilde u}))}=w^{-1}\circ\tilde\rho\circ w. 
$$
By the unique continuation theorem, $\tilde\rho\circ\underline{\tilde u}=\underline{\tilde u}\circ \underline I$  holds on $S^2\setminus \underline\Gamma$ and $\underline I$ is still anticonformal and involutive.  In a similar vein we have $p\circ I=\underline I\circ p$ since
$$
\underline{\tilde u}(\underline I\circ p(z))=\tilde\rho\circ\underline{\tilde u}(p(z))=\tilde\rho\circ\tilde u(z)=\tilde u(I(z))=\underline{\tilde u}(p\circ I(\bar z))
$$
and $\tilde{\underline u}$ is an embedding almost everywhere.  Since $\underline I$ is anticonformal and proper, it is a complex polynomial composed with the complex conjugation $I$. Moreover involutivity yields that $\underline I$ is either $I$ or $-I$. If the former is the case, we are done. Otherwise, namely $\underline I=-I$, we repeat the argument with $\underline{\tilde u}$ replaced by $\underline{\tilde u}\circ i$. This finishes the proof.
\end{proof}

In the absence of symmetry, transversality results and Fredholm index computations of somewhere injective finite energy spheres are discussed in various articles, see \cite{Dra04,Bou02, Bou06,Wen14}. It is well known that these arguments are easily modified to prove the following statements. Note that the Fredholm index of $\tilde u_I$ is derived from that of the associated invariant finite energy plane $\tilde u$ by taking half of the constant term in the index formula of $\tilde u$ and replacing the Conley-Zehnder index of symmetric periodic orbits by the Robbin-Salamon index of the corresponding half-chords.

\begin{Thm}\label{thm:trasversality}
Let $\tilde u_I:(D_+,\Fix I)\setminus\Gamma_I\to(\Xi,\Fix\tilde\rho)$ be a somewhere injective finite energy half-sphere with nondegenerate asymptotic orbits. For a generic $\wt J\in\JJ_{\tilde\rho}$, $\tilde u_I$ is regular. Moreover the Fredholm index of $\tilde u_I$ is 
\bean
\mathrm{Ind}(\tilde u_I)&=\sum_{z\in\Gamma_\p^+}\mu^\Phi_{\RS}(C_z)+\sum_{z\in\Gamma_o^+}\mu_\CZ^\Phi(P_z)-\sum_{z\in\Gamma_\p^-}\mu_\RS^\Phi(C_z)
-\sum_{z\in\Gamma_\p^-}\mu_\CZ^\Phi(P_z)\\
&\quad+\Big(\frac{\dim\Xi}{2}-3\Big)\frac{(2-\#\Gamma_\p-2\#\Gamma_o)}{2}.
\eea
for a unitary trivialization $\Phi$ of $(\tilde u_I)^*(T\Xi,T\Fix \tilde\rho)$. If a somewhere injective finite energy sphere $\tilde u:S^2\setminus\Gamma\to\Xi$ with nondegenerate asymptotic orbits is not invariant, it is regular for a generic $\wt J\in\JJ_{\tilde\rho}$ with the Fredholm index
$$
\mathrm{Ind}(\tilde u)=\sum_{z\in\Gamma^+}\mu_\CZ^\Phi(P_z)-\sum_{z\in\Gamma^-}\mu_\CZ^\Phi(P_z)
+\Big(\frac{\dim\Xi}{2}-3\Big)(2-\#\Gamma)
$$
for a unitary trivialization $\Phi$ of $\tilde u^*T\Xi$. Note that this transversality result ensures that 
$$
\Ind(\tilde u_I),\quad \Ind(\tilde u)\geq0
$$
in general and furthermore 
$$
\Ind(\tilde u_I),\quad\Ind(\tilde u)\geq1
$$ 
if $(\Xi,\Om,\tilde\rho)$ is the symplectization of a contact manifold $(M,\alpha,\rho)$ and $\pi\circ T\tilde u_I,\,\pi\circ T\tilde u\neq 0$.
\end{Thm}


\section{Real holomorphic curves in $(\mathbb{CP}^2,\mathbb{RP}^2)$}

In the following we think of $\mathbb{CP}_\infty^1$ as the line at infinity in $\mathbb{CP}^2$ via the embedding
$[z_0,z_1] \mapsto [z_0,z_1,0]$. The complement $\mathbb{CP}^2 \setminus \mathbb{CP}_\infty^1$ can now be identified with $\C^2$ via the map $[z_0,z_1,z_2] \mapsto \big(\frac{z_0}{z_2},\frac{z_1}{z_2}\big)$. We endow $\mathbb{CP}^2$ with the Fubini-Study form $\om_{\FS}$.  We consider on $(\mathbb{CP}^2,\om_{\FS})$ the antisymplectic involutions 
$$
\wh\rho \colon \mathbb{CP}^2 \to \mathbb{CP}^2,\quad [z_0,z_1,z_2]\mapsto[\bar{z}_0,\bar{z}_1,\bar {z}_2]
$$
so that $\wh\rho|_{\C^2}=\tilde\rho:(z_1,z_2)\mapsto (\bar z_1,\bar z_2)$.
The fixed point set of an antisymplectic involution is a Lagrangian submanifold unless it is empty. In particular, we have
$$
\Fix\wh\rho=\mathbb{RP}^2 \subset \mathbb{CP}^2.
$$
As in the complex case we think of $\mathbb{RP}^2$ as a compactification of $\R^2$ by adding a circle $\mathbb{RP}^1$ at infinity. We choose a 1-form 
$$
\lambda_\FS:=\frac{1}{2(1+\sum_{j=1}^2( x_j^2+y_j^2))}\sum_{i=1}^2(x_idy_i-y_idx_i)
$$
on $\C^2$, where $z_i=x_i+iy_i$ so that 
$$
d\lambda_\FS=\om_{\FS}|_{\C^2}=\frac{1}{(1+\sum_{j=1}^2 (x_j^2+y_j^2))^2}\sum_{i=1}^2dx_i\wedge dy_i.
$$
Let $M$ be a starshaped hypersurface in $(\C^2,\om=d\lambda)$ such that $\tilde\rho(M)=M$. We additionally assume that $(M,\alpha=\lambda|_M)$ is nondegenerate, i.e. every periodic Reeb orbit is nondegenerate. We choose $\kappa>0$ such that $\frac{1}{\kappa}M:=\{\frac{1}{\kappa}(z_0,z_1)\in\C^2\,|\,(z_0,z_1)\in M\}$ is included in the open unit ball $B\subset\C^2$.   Since 
$$
\SS:(B,\om)\to(\C^2,\om_\FS),\quad (x_1,y_1,x_2,y_2)\mapsto\frac{1}{\sqrt{1-\sum_{j=1}^2( x_j^2+y_j^2) }}(x_1,y_1,x_2,y_2)
$$
is a symplectomorphism such that $\SS\circ\tilde\rho=\tilde\rho$ and $\SS^*\lambda_\FS=\lambda|_B$, the dynamics of $\lambda_\FS$ on $\big(\SS(\frac{1}{\kappa}M),\tilde \rho\big)$ is equivalent to that of $\lambda$ on $(M,\tilde\rho)$. Abusing the notation, we denote $\SS(\frac{1}{\kappa}M)\subset\mathbb{CP}^2$ by $M$ again.  Let $L$ be a Liouville vector field defined on a neighborhood of $M\subset\mathbb{CP}^2$, i.e. $i_L\om_\FS=\lambda_\FS$. Using the flow of $L$ we identify a neighborhood of $M$ with $([-\epsilon,\epsilon]\x M,e^r\lambda_\FS|_M)$ for small $\epsilon>0$ where $r$ where $r$ is the coordinate on $[-\epsilon,\epsilon]$. The hypersurface $M$ divides $\mathbb{CP}^2$ into two connected compact components with boundary $M$. We denote by $V$ the component containing $\mathbb{CP}^1_\infty$ and by $W=\mathbb{CP}^2\setminus \mathring V$.
Following \cite{HWZ03}, we stretch the neck of $\mathbb{CP}^2$ in an open neighborhood of $M$ and obtain $(\mathbb{CP}^2_N,\om_N)$, $N\in\N$ such that 
\begin{itemize}
\item [i)] $\mathbb{CP}^2_N$ is diffeomorphic to $W\sqcup([-N,N]\x M)\sqcup V/\sim$ where $\sim$ indicates the boundary identification $\p W=\{-N\}\x M$ and $\p V=\{N\}\x M$;
\item [ii)] the symplectic form $\om_N$ is defined by 
$$
\om_N=\left\{\begin{array}{ll} d(\varphi_N\lambda_\FS) \quad & [-N-\epsilon,N+\epsilon]\x M,\\[1ex]
\om_{\FS} & W\sqcup V\setminus ([-N-\epsilon,-N]\cup[N,N+\epsilon])\x M.
\end{array}\right.
$$
\end{itemize}
for a small $\epsilon>0$ where $\varphi$ is a smooth function such that 
$$
\varphi_N:[-N-\epsilon,N+\epsilon]\to\R,\quad\varphi_N'>0,\quad \varphi_N(r)=\left\{\begin{array}{ll} e^{r+N}\quad & [-N-\epsilon,-N-\tfrac{\epsilon}{2}],\\[1ex]
e^{r-N} & [N+\tfrac{\epsilon}{2},N+\epsilon].
\end{array}\right.
$$
For the detailed construction we refer the reader to \cite{HWZ03}.

We also stretch $\hat\rho$ and the complex structure to $\hat\rho_N$ and $\wh J_N$ respectively so that 
\begin{itemize}
\item [i)] $\hat\rho_N=\hat\rho$ on $V\sqcup W$ and $\hat\rho_N=\tilde\rho$ on $[-N,N]\x M\subset\C^2$. Accordingly, $\hat\rho_N^*\om_N=-\om_N$.
\item [ii)] $\wt J_N$ preserves the contact hyperplane $\xi=\ker\lambda_\FS$ and $\wt J_N|_{\xi}$ is $\om_N|_{\xi}$-compatible.
\item [iii)] $\wh J_N$ is antiinvariant with respect to  $\hat\rho_N$, equal to the standard complex structure of $\mathbb{CP}^2$ near $\mathbb{CP}^1_\infty$, and cylindrical on the cylindrical part $[-N,N]\x M$.
\end{itemize}
We abbreviate by $\JJ_N$ the set of $\om_{\FS}$-compatible almost complex structures with these three properties. Denote by $\mathbb{RP}^2_N:=\Fix \hat\rho_N$.

We are interested in $\wh J_N$-holomorphic spheres in $\mathbb{CP}_N^2$ invariant under $\hat\rho_N$. Let $C$ be a $\wh J_N$-holomorphic sphere homologous to $\mathbb{CP}^1\subset \mathbb{CP}_N^2$. Then by the adjunction formula \cite{Gro85,McD91}, $C$ is always embedded. There exists a unique complex line $\mathbb{CP}^1\subset(\mathbb{CP}^2,i)$ passing through any two prescribed points $\mathbb{CP}^2$. Thus, there exists a unique $\wh J_N$-holomorphic sphere homologous to $(\mathbb{CP}^1_N,\wh J_N)$ passing through any two points $p,\,q\in\mathbb{CP}_N^2$ due to positivity of  intersections and  the implicit function theorem, see \cite{HWZ03} for details. This is true even for $\wh J_N\in\JJ_N$ due to automatic transversality \cite{HLS98}. Furthermore, observe that if $p,\,q\in\mathbb{RP}^2_N$, such a $C$ is $\hat\rho_N$-invariant, i.e. 
$$
\hat\rho_N(C)=C
$$
by the uniqueness.

We fix a point $o_\infty$ in $\mathbb{RP}^1_\infty\subset\mathbb{RP}^2_N$. Then \cite[Theorem 2.15]{HWZ03} refines as follows.
\begin{Thm}\label{thm:invariant holomorphic sphere}
There exists a unique embedded $\wh J_N$-holomorphic sphere 
$$
C^q_N\subset (\mathbb{CP}^2_N,\wh J_N,\hat\rho_N,\om_N)
$$
for $q\neq o_\infty$ which is homologous to $\mathbb{CP}^1\subset \mathbb{CP}^2_N$ and passes through $q,\,o_\infty\in \mathbb{CP}^2_N$. Note that 
\begin{itemize}
\item [i)] $C^{q}_N$ and $C^{p}_N$ either coincide or intersect exactly at $o_\infty$ transversally.
\item [ii)] $\hat\rho_N(C^q_N)=C^{\hat\rho_N(q)}_N$ is also an embedded $\wh J_N$-holomorphic sphere.
\item [iii)] $C^q_N$ is $\hat\rho_N$-invariant if $q\in\mathbb{RP}^2_N$.
\item [iv)] $C^q_N=\mathbb{CP}^1_\infty$ if $q\in\mathbb{CP}^1_\infty$.
\end{itemize}
\end{Thm}
Observe that $\{C^q_N\}$ form a singular foliation of $\mathbb{CP}^2_N$ with the only singular point $o_\infty$. Moreover $C^{q_1}_N$ with $q_1\in \mathbb{RP}^2_N\cap(\{-N\}\x M)$ goes below the level $\{0\}\x M$ but $C^{q_2}_N$ with $q_2\in\mathbb{RP}^1_\infty$ does not, see item iv) of Theorem \ref{thm:invariant holomorphic sphere}. Since $\{C^q_N\,|\, q\in\mathbb{RP}^2_N\}$ is a continuous family of ($\hat\rho_N$-invariant) leaves of the foliation, there exists $q_0\in\mathbb{RP}^2_N\cap(([0,-N]\x M)\sqcup V/\sim)$ such that 
$$
C^{q_0}_N\subset ([-N,N]\x M)\sqcup V/\sim,\quad \min \pi_N(C^{q_0}_N\cap([-N,N]\x M)) =0
$$
where $\pi_N:[-N,N]\x M\to [-N,N]$ is the projection along $M$.  This special $\wh J_N$-holomorphic sphere will generate an invariant fast finite energy plane.

\section{Invariant fast finite energy plane}

In this section we briefly discuss compactness results of a sequence of holomorphic parametrizations of $C_N^q$, $N\in\N$ and show that a piece of the limit holomorphic curves is indeed an invariant fast finite energy plane. Recall that we have assumed that $(M,\alpha)$ is nondegenerate. To begin with, the limiting object of $(\mathbb{CP}^2_N,\wh J_N,\hat\rho_N,\om_N)$ is 
$$
(\wt W,\wh J_\infty,\hat\rho_\infty,\om_{\wt W} ),\quad(\R\x M,\wh J_\infty,\hat\rho_\infty,d\lambda),\quad (\wt V,\hat  J_\infty,\hat\rho_\infty,\om_{\wt V})
$$
where
$$
\wt W=W\sqcup([0,\infty)\x M)/\sim,\quad \wt V=((-\infty,0])\x M)\sqcup V/\sim
$$
and 
$$
\om_{\wt W}=\left\{\begin{array}{cc}\om_{\FS} & W\setminus [-\epsilon,0]\x M\\[1ex]
d(\varphi_+\alpha) & [-\epsilon,\infty)\x M
\end{array}\right.,\quad 
\om_{\wt V}=\left\{\begin{array}{cc}\om_{\FS} & V\setminus [0,\epsilon]\x M\\[1ex]
d(\varphi_-\alpha) & (-\infty,\epsilon]\x M
\end{array}\right.
$$
for 
$$
\varphi_+:[-\epsilon,\infty)\to(0,1),\quad \varphi_+'>0,\quad  \varphi_+(s)=e^s\;\textrm{ near } -\epsilon
$$ 
and 
$$
\varphi_-:(-\infty,\epsilon]\to(0,\infty),\quad \varphi_-'>0,\quad  \varphi_-(s)=e^s\;\textrm{ near }\; \epsilon.
$$ 
Here $\sim$ indicates the identification as before along the boundaries which are diffeomorphic to a copy of $M$. Note that 
$$
\hat\rho_\infty|_{\wt W}=\tilde\rho,\qquad \hat\rho_\infty|_{\R\x M}=\tilde\rho.
$$

We choose a $\wh J_N$-holomorphic parametrization 
$$
w_N:S^2\to \mathbb{CP}^2_N,\quad w_N\circ I=\hat\rho_N\circ w_N,\;\quad w_N(z_0)\in\{0\}\x M,\; w_N(\infty)=o_\infty.
$$ 
of the embedded $\wh J_N$-holomorphic sphere $C_N^{q_0}$ in Theorem \ref{thm:invariant holomorphic sphere} so that the sequence $\{w_N\}$  converges to a holomorphic building as follows. Here $I(z)=\bar z$, $z\in S^2=\C\cup\{\infty\}$. The bottom of the building is composed of finite energy planes in $\wt W$ or $\R\x M$
$$
\tilde u_{1,1}:\C\pf\wt W\;(\textrm{or $\R\x M$}),\;\cdots,\;\tilde u_{1,n_1}:\C\pf\wt W\;(\textrm{or $\R\x M$})
$$
for $n_1\in\N$,
$$
\tilde u_{i,1}:S^2\setminus\Gamma_{i,1}\pf \R\x M,\;\cdots,\;\tilde u_{i,n_i}:S^2\setminus\Gamma_{i,n_i}\pf \R\x M
$$  
for $n_i\in\N$ are punctured finite energy spheres in the middle stories $2\leq i\leq q-1$ and the top is a single punctured finite energy sphere
$$
\tilde u_{q,1}=\tilde u_q:S^2\setminus\Gamma_q\pf \wt V
$$
meeting the following properties.
\begin{itemize}
\item[i)] The positive asymptotic periodic orbits of $\tilde u_{i,j}$ match with the negative asymptotic periodic orbits of $\tilde u_{i+1,j'}$, $1\leq i\leq q-1$ appropriately.
\item[ii)] Every curve $\tilde u_{i,j}$, $1\leq i\leq q-1$, $1\leq j\leq n_i$ has precisely one positive puncture. For every $i$, there exists $j$ such that $\tilde u_{i,j}$ is not a trivial cylinder over a periodic orbit.
\item[iii)] The whole building is invariant, i.e. if $\tilde u_{i,j}$ intersects with $\Fix \hat\rho_\infty$, $\tilde u_{i,j}\circ I=\hat  \rho_\infty\circ\tilde u_{i,j}$ and otherwise $\tilde u_{i,j}\circ I=\hat\rho_\infty\circ\tilde u_{i,j'}$ for some $1\leq j\neq j'\leq n_i$. A curve $\tilde u_{i,j}$ is invariant if and only if the positive asymptotic periodic orbit is symmetric.
\item[iv)] In particular, $\tilde u_q$ is invariant and intersects $\mathbb{CP}^1_\infty$ once at $o_\infty$ transversally.
\item[v)] There exist $1\leq k\leq q-1$ and $1\leq \ell\leq n_k$ such that $\Gamma_{k,\ell}=\{\infty\}$ and 
$$
\tilde u_{k,\ell}=(a_{k,\ell},u_{k,\ell}):S^2\setminus\Gamma_{k,\ell}=\C\to\R\x M,\quad  \min a_{k,\ell}=0.
$$
\end{itemize}
We refer to \cite{HWZ03} for details on this SFT compactness result, see also \cite{BEHWZ03,CM05}. In particular see \cite[Proposition 7.1]{HWZ03} for item v).

Recall that $\mu_\CZ(P)=\mu^{\Phi}_\CZ(P)$ and $\mu_\RS(C)=\mu^{\Phi}_\RS(C)$ for any trivialization $\Phi$ of filling (half-) disks of a periodic orbit $P$ and a chord $C$ in $(M,X,\rho)$, see \eqref{eq:index notation}.

\begin{Lemma}\label{Lemma:Fredholm indices}
Let  $\tilde u_q:S^2\setminus\Gamma_q\to\wt V$ be the invariant finite energy sphere above. The constrained Fredholm index of the finite energy half-sphere $(\tilde u_q)_I:(D_+,\Fix I)\setminus(\Gamma_q)_I\to(\wt V,\Fix\hat\rho_\infty)$ with $(\tilde u_q)_I(\infty)=o_\infty$ equals 
$$
 \Ind((\tilde u_q)_I;o_\infty)=-\sum_{z\in(\Gamma_{q})_\p}\mu_{\RS}(C_z)-\sum_{z\in(\Gamma_{q})_o}\mu_\CZ(P_z)+\frac{\#\Gamma^-_\p}{2}+\#\Gamma^-_o+1
$$
where $C_z$'s and $P_z$'s are asymptotic chords and periodic orbits of $\tilde u_q$ at $z\in(\Gamma_q)_I=(\Gamma_{q})_\p\sqcup(\Gamma_{q})_o$ respectively.
\end{Lemma}
\begin{proof}
This corresponds to \cite[Proposition 5.2]{HWZ03} and we outline a proof. According to Theorem \ref{thm:trasversality}, for  a symmetric unitary trivialization $\Phi$ of $\tilde u_q^*T\wt V$,
$$ 
\Ind((\tilde u_q)_I;o_\infty)=-\sum_{z\in\Gamma^-_\p}\mu^\Phi_\RS(C^-_z)-\sum_{z\in\Gamma^-_o}\mu^\Phi_\CZ(P_z^-)+\frac{\#\Gamma^-_\p}{2}+\#\Gamma^-_o-2
$$
since the condition $(\tilde u_q)_I(\infty)=o_\infty$ decreases the index by 1. We choose maps from $\C$ to $\R\x M$ capping off punctures of $\tilde u_q$ so that these together with $\tilde u_q$ give a map $S^2\to\mathbb{CP}^2_N$ homologous to $\mathbb{CP}^1\subset\mathbb{CP}_N^2$ after gluing. Then the arguments in Corollary \ref{cor:index+trivialization} show
$$
\sum_{z\in(\Gamma_{q})_\p}\mu_{\RS}(C_z)+\sum_{z\in(\Gamma_{q})_o}\mu_\CZ(P_z)=\sum_{z\in(\Gamma_{q})_\p}\mu^\Phi_{\RS}(C_z)+\sum_{z\in(\Gamma_{q})_o}\mu^\Phi_\CZ(P_z)+c_1(T\mathbb{CP}^2)[\mathbb{CP}^1]
$$
and this proves the lemma.
\end{proof}

\begin{Prop}\label{prop:existence of invariant finite energy plane}
If $(M,\alpha)$ is dynamically convex, the finite energy plane $\tilde u_{k,\ell}:\C\to\R\x M$ is somewhere injective and invariant. If we write $P$ for the symmetric asymptotic periodic orbit and $C$ for the half-chord of it, we have
$$
\mu_\RS(C)=\frac{3}{2},\qquad \mu_\CZ(P)\in\{3,4\}.
$$
\end{Prop}
In fact we have shown here that nonsymmetric asymptotic periodic orbits in the holomorphic building have Conley-Zehnder indices in $\{1,2\}$ and the half-chords of symmetric asymptotic periodic orbits have Robbin-Salamon index in $\{\frac{1}{2},\frac{3}{2}\}$ without dynamical convexity. However this fact will not be used later.
\begin{proof}
The idea of the proof is analyzing indices using Corollary \ref{cor:index behaviors}, Theorem \ref{thm:trasversality}, and Lemma \ref{Lemma:Fredholm indices} from the bottom to the top of the holomorphic building inductively. This is a symmetric counterpart of \cite[Proposition 5.7]{HWZ03}.

Let $\tilde {\underline{u}}_{1,j}$ be the underlying somewhere injective finite energy plane of $\tilde u_{1,j}$. If $\tilde u_{1,j}$ is not invariant, $\tilde {\underline{u}}_{1,j}$ is cut out transversally for a generic $\wh J_\infty\in\JJ_{\tilde\rho}$, its positive asymptotic periodic orbit has Conley-Zehnder index at least 1. If it is invariant, the finite energy half-plane $(\tilde {\underline{u}}_{1,j})_I$ is cut out transversally for a generic $\wh J_\infty\in\JJ_{\tilde\rho}$ and the asymptotic chord has Robbin-Salamon index at least $\frac{1}{2}$. Recall that $\mu_\CZ$ and $\mu_\RS$  do not decrease under iteration if the initial indices are at least $1$ and $\frac{1}{2}$ respectively. Moreover otherwise, they do not increase.

Suppose that $\tilde u_{i,j}$, $2\leq i\leq q-1$ is invariant and Conley-Zehnder indices of its nonsymmetric negative asymptotic periodic orbits are at least 1 and Robbin-Salamon indices of the half-chords of its symmetric negative asymptotic periodic orbits are at least $\frac{1}{2}$. We denote by $\underline{\tilde u}_{i,j}$ the underlying somewhere injective finite energy sphere of $\tilde u_{i,j}$. Then its negative asymptotic orbits have the same property as observed above. If $\pi\circ T u_{i,j}\neq0$, since 
$$
\mathrm{Ind}\big((\underline{\tilde u}_{i,j})_I\big)=\mu_\RS(C^+)-\sum_{z\in\Gamma^-_\p}\mu_\RS(C^-_z)-\sum_{z\in\Gamma^-_o}\mu_\CZ(P^-_z)-\frac{1-\#\Gamma^-_\p-2\#\Gamma^-_o}{2}\geq 1,
$$
where $C^+$ a positive asymptotic chord, $C_z^-$'s are negative asymptotic chords, and $P_z^-$'s are negative asymptotic periodic orbit of $(\underline{\tilde u}_{i,j})_I$, we have $\mu_\RS(C^+)\geq \frac{3}{2}$. In the case that $\pi\circ T \tilde u_{i,j}=0$, $\underline{\tilde u}_{i,j}$ is a trivial cylinder, one can easily show that $\mu_\RS(C^+)\geq\frac{1}{2}$. If $\tilde u_{i,j}$ is not invariant, the argument of \cite[Proposition 5.7]{HWZ03} goes through and thus the Conley-Zehnder index of the positive asymptotic periodic orbit of $\tilde u_{i,j}$ is at least 1. Hence we have shown that asymptotic nonsymmetric periodic orbits of the limit building have Conley-Zehnder index at least 1 and the half-chords of asymptotic symmetric periodic orbits have Robbin-Salamon index at least $\frac{1}{2}$ by induction.

Suppose that there is a symmetric asymptotic periodic orbit in the holomorphic building such that its half-chord has  Robbin-Salamon index at least $\frac{5}{2}$. Then arguing as above, $(\tilde u_q)_I$ has a negative asymptotic chord with index at least $\frac{5}{2}$. Since $(\tilde u_q)_I$ is somewhere injective due to the constraint at $o_\infty$, by Lemma \ref{Lemma:Fredholm indices} we have
$$
\mathrm{Ind}((\tilde u_q)_I;o_\infty)=-\sum_{z\in\Gamma^-_\p}\mu_\RS(C^-_z)-\sum_{z\in\Gamma^-_o}\mu_\CZ(P_z^-)+\frac{\#\Gamma^-_\p}{2}+\#\Gamma^-_o+1\leq-1
$$
where $C_z^-$'s are negative asymptotic chords and $P_z^-$ are negative asymptotic periodic orbits. This contradiction shows that the half-chords of symmetric asymptotic periodic orbits in the holomorphic building have Robbin-Salamon index in $\{\frac{1}{2},\frac{3}{2}\}$. In a similar vein, every nonsymmetric asymptotic periodic orbit in the holomorphic building has to have  Conley-Zehnder index in $\{1,2\}$, since otherwise $\tilde u_q$ has a negative nonsymmetric asymptotic periodic orbit with Conley-Zehnder index at least 3 or $(\tilde u_q)_I$ has a negative asymptotic chord with index at least $\frac{5}{2}$ which in turn imply $\mathrm{Ind}(\tilde w_I;o_\infty)\leq -1$ again.

Suppose by contradiction that the finite energy plane $\tilde u_{k,\ell}$ is not invariant. Let $\underline{\tilde u}_{k,\ell}$ be the underlying somewhere injective finite energy plane of $\tilde u_{k,\ell}$. Then its asymptotic periodic orbit $P$ is nonsymmetric and has Conley-Zehnder index at least 3 due to the dynamically convexity assumption and Corollary \ref{cor:index behaviors}. Arguing as before again, this results in the curve in the next story and hence $\tilde u_q$ has to have either a nonsymmetric negative asymptotic periodic orbit of Conley-Zehnder index at least 3 or a symmetric one with Robbin-Salamon of its half-chord at least $\frac{5}{2}$. We showed that neither happens. This concludes that $\tilde u_{k,\ell}$ has to be invariant.

To show that $\tilde u_{k,\ell}$ is somewhere injective, suppose $\tilde u_{k,\ell}=\underline{\tilde u}_{k,\ell}\circ p$ for some holomorphic branched covering $p:\C\to\C$ and some somewhere injective invariant finite energy plane $\underline{\tilde u}_{k,\ell}$. If we write the positive asymptotic chord of $({\tilde u}_{k,\ell})_I$  by $C$ and that of $(\underline{\tilde u}_{k,\ell})_I$ by $C_0$, $C=(C_0)^{\deg p}$. Since $\mu_\RS(C_0)\geq\frac{3}{2}$ by Proposition \ref{Prop:relative wind and index} (or by Corollary \ref{cor:index behaviors} again) and $\mu_\RS(C)\leq \frac{3}{2}$ as shown above, by Proposition \ref{Prop:RS-index increases} $\deg p=1$ and therefore ${\tilde u}_{k,\ell}$ is somewhere injective with $\mu_\RS(C)=\frac{3}{2}$. The claim that $\mu_\CZ(P)\leq4$ follows from Proposition \ref{prop:index relation} and Proposition \ref{prop:Hormander index}.

\end{proof}
For notational convenience, henceforth, we write
$$
\tilde u:=\tilde u_{k,\ell}=(a,u):\C\pf\R\x M.
$$ 
In what follows, we will show that this somewhere injective finite energy plane is fast, i.e. the asymptotic periodic orbit is nondegenerate and simple, and if $u$ is an immersion and transversal to $X$. The last requirement is equivalent to  $\wind_\pi(\tilde u)=0$.

\begin{Prop}
The invariant finite energy plane $\tilde u=(a,u):\C\to\R\x M$ satisfies,
$$
\wind_\pi(\tilde u)=0,\quad \wind_\infty(\tilde u)=1.
$$
\end{Prop}
\begin{proof}
Since $\wind_\pi(\tilde u)\geq 0$ it follows that $\wind_\infty(\tilde u)\geq1$. From Proposition \ref{Prop:relative wind and index} we have 
$$
2\wind_\infty(\tilde u_I)\leq\mu_{\RS}(C)-\frac{1}{2}=1
$$
where $C$ is the asymptotic chord of $\tilde u_I$. Therefore we conclude with  Proposition \ref{Prop:wind=2 relative wind} that $\wind_\pi(\tilde u)=0$ and $\wind_\infty(\tilde u)=1$.
\end{proof}

\begin{Cor}\label{cor:fast}
The invariant finite energy plane $\tilde u$ is embedded and fast.
\end{Cor}
\begin{proof}
Due to dynamical convexity, the asymptotic periodic orbit of $\tilde u$ is simple since otherwise $\mu_\CZ(P)\geq5$ for the asymptotic periodic orbit $P$ of $\tilde u$ by Corollary \ref{cor:index behaviors}. This together with the previous proposition shows that $\tilde u$ is fast. Note that $\tilde u$ is an embedding near infinity due to Theorem \ref{thm:asymptotic formula} together with the fact that the asymptotic periodic orbit of $\tilde u$ is simple. Moreover, since it is a piece of the limit of embedded curves, the whole $\tilde u$ is an embedding.
\end{proof}

A standard argument in \cite{HWZ98,Hry12} shows that the existence of a fast finite energy plane into $\R\x M$ gives rise to a global disk-like surface of section for the Reeb vector field $X$ in $M$ under nondegeneracy and dynamical  convexity. Moreover nondegeneracy can be dropped by a limiting argument, see \cite{HWZ98,Hry11}. Hence our invariant fast finite energy plane $\tilde u$ provides an invariant global disk-like surface of section in Theorem \ref{thm:invglobal}. Other assertions in the theorem follow as by-products. For the reader's convenience, we outline this argument.

We say that two finite energy planes $\tilde u,\,\tilde u':\C\to\R\x M$ are equivalent if there exists a biholomorphic transformation $\varphi$ on $\C$ such that $\tilde u=\tilde u'\circ \varphi$. We assume again that $(M,X)$ is nondegenerate.  Let $\MM$ be the moduli space of equivalence classes of finite energy planes in $\R\x M$ equipped with the $C^\infty_\mathrm{loc}$-topology and the topology of uniform convergence near infinity. 
Consider the connected component $\MM(\tilde u)$ of $\MM$ containing $\tilde u$. Denote by $P$ the asymptotic periodic orbit of $\tilde u$ and by $A_P$ the asymptotic operator. For any $c\in(-\infty,0)$ such that $c\notin\sigma(A_P)$, the constrained moduli space $\MM(\tilde u;c)$ is  composed of elements in $\MM(\tilde u)$ with the asymptotic convergence rate $\lambda$ in Theorem \ref{thm:asymptotic formula}, which is a negative eigenvalue of the asymptotic operator $A_P$, smaller than $c$. The virtual dimension of $\MM(\tilde u;c)$ agrees with the constrained Fredholm index 
$$
\Ind(\tilde u;c)=\mu_{\CZ}(P;c)-1.
$$
Here the constrained Conley-Zehnder index is defined by
$$
\mu_{\CZ}(P;c):=2\alpha(S_P;c)+p(S_P;c)
$$
where 
$$
\alpha(S_P;c):=\max\{w(\lambda,S_P)\,|\,\lambda\in\sigma(A_{S_P})\cap(-\infty,c)
$$
and 
$$
p(S_P):=\min\{w(\lambda,S_P)\,|\,\sigma(A_{P})\cap(0,\infty)\}-\alpha(S_P;c),
$$
or equivalently
\bean
\mu_{\CZ}(P;c)&=\mu_{\CZ}(P)-\#\{\sigma(A_P)\cap(c,0))\}\\
&=\max \{2i+j\,|\,\lambda_i^j(S_P)< c\}.
\eea
Note that $\lambda_1^0(S_P)<0$ since there is a finite energy plane in the symplectization asymptotic to $P$, see Proposition \ref{Prop:wind and index}. We choose $\delta<0$ satisfying
\beq\label{eq:delta}
\delta\in(\lambda_1^1(S_P),\lambda_2^0(S_P))\cap(-\infty,0)
\eeq
if the eigenvalue $\lambda_1^1(S_P)<0$ or otherwise
$$
\delta\in(\lambda_1^0(S_P),0)
$$
so that $\MM(\tilde u;\delta)$ consists only of fast finite energy planes in the component $\MM(\tilde u)$.
In particular, if $\mu_{\CZ}(P)$ is either 2 or 3, $\MM(\tilde u)=\MM(\tilde u;\delta)$. From Corollary \ref{cor:fast},
$$
\tilde u\in\MM(\tilde u;\delta).
$$
We denote by 
$$
\#\Gamma_{\mathrm{even}}(\tilde u;\delta):=\frac{1}{2}\big(1+(-1)^{\mu_{\CZ}(P;\delta)}\big) \in\{0,1\}.
$$ 
Observe that
\beq\label{eq:Gammaeven}
\Ind(\tilde u;\delta)+\#\Gamma_{\mathrm{even}}(\tilde u;\delta)=2.
\eeq

\begin{Lemma}\label{lemma:nicely embedded}
For any $\tilde u'=(a',u')\in\MM(\tilde u,\delta)$, $u':\C\to M$ is an embedding.
\end{Lemma}

This can be proved by following \cite[Theorem 2.7]{HWZ03} closely using the fastness property. However by now this directly follows from Siefring's work \cite{Sie11}.

\begin{proof}
Although the assertion immediately follows from \cite[Theorem 2.6]{Sie11}, we explain how
the assertion follows from the adjunction formula by computing the (constrained) generalized intersection number $[\tilde u;\delta]*[\tilde u;\delta]$. According to \cite[Theorem 2.3, Corollary 4.7]{Sie11} (see also \cite{Wen10}), 
$$
[\tilde u;\delta]*[\tilde u;\delta]=2\,\mathrm{sing}(\tilde u)+\frac{1}{2}(\Ind(\tilde u;\delta)-2+\#\Gamma_{\mathrm{even}}(\tilde u;\delta))-1+\bar\sigma(P).
$$
Here $\mathrm{sing}(\tilde u)$ is the singularity index, and $\bar\sigma(P)$ the spectral covering number. The facts that $\tilde u$ is an embedding and the asymptotic periodic orbit $P$ is simple imply $\mathrm{sing}(\tilde u;c)=0$ and $\bar\sigma(P)=1$. Therefore we deduce from (\ref{eq:Gammaeven}) that $[\tilde u;\delta]*[\tilde u;\delta]=0$.  Due to homotopy invariance of the intersection number, 
$$
[\tilde u';\delta]*[\tilde u_r';\delta]=[\tilde u;\delta]*[\tilde u;\delta]=0
$$ 
where $\tilde u_r'(z)=(a'(z)+r,u'(z))$ for any $r\in\R$ and thus $u':\C\to S^3$ is injective for every $\tilde u'=(a',u')\in\MM(\tilde u,\delta)$. Since $\tilde u'$ is fast, $u'$ is an immersion and hence an embedding. 
\end{proof}

Suppose that a nondegenerate starshaped hypersurface $(M,\alpha)$ is dynamically convex. In fact, Proposition \ref {prop:automatic transversality} holds in this constrained case and thus $\MM(\tilde u;\delta)$ is a smooth manifold of dimension 2. Since the constrained moduli space $\MM(\tilde u;\delta)$ consists of equivalence classes of fast finite energy planes, the quotient space $\MM(\tilde u;\delta)/\R$ dividing out  the translation in the $\R$-direction of the symplectization $\R\x M$ is compact due to Hryniewicz \cite{Hry11} and hence diffeomorphic to $S^1$. Such a  $S^1$-family of planes are embedded due to  Lemma \ref{lemma:nicely embedded} in $M$ and form a holomorphic open book decomposition due to Fredholm theory. Here we mean by holomorphic that it is liftable to an $S^1$-family of finite energy planes to the symplectization.  As shown by Hofer, Wysocki, and Zehnder in \cite[Proposition 5.1]{HWZ98}, every page of the open book decomposition is a global disk-like surface of section in particular our invariant fast finite energy plane $\tilde u=(a,u):\C\to\R\x M$ gives rise to an invariant global disk-like surface of section $\overline{u(\C)}\subset M$ for $X$. A limiting argument in \cite{HWZ98,Hry11} enables us to remove the nondegeneracy assumption. This completes the proof of Theorem \ref{thm:invglobal}.

\begin{Rmk}\label{rmk:symmetric open book decomposition}
({\em Symmetric open book decomposition}). If we do not require holomorphicity of an open book decomposition, we are easily able to construct an invariant open book decomposition out of an invariant global surface of section. Let $M$ be a manifold of dimension 3 with a vector field $X$ carrying an involution $\rho$ such that $\rho_*X=-X$. Suppose that there exists  a  global surface of section $\Sigma$ in $M$ invariant under $\rho$. 
Let $\tau_x$ be the first return time of $x\in\mathring\Sigma$, i.e.
$$
\tau_x:=\min\{t>0\,|\,\phi_X^{t}(x)\in\mathring\Sigma\}.
$$
Note that by $\phi_X^t\circ \rho=\rho\circ \phi_X^{-t}$, 
$$
\tau_x=\tau_{\rho\circ\phi_X^{\tau_x}(x)}.
$$
We define a diffeomorphism $\Phi:S^1\x\Sigma\to M\setminus\p\Sigma$ as 
$$
\Phi(\theta,z):=\phi_X^{\theta\tau_{\iota(z)}}\circ\iota(z)
$$
where $\iota:\Sigma\to M\setminus \p\Sigma$ is the embedding. Then every page $\Phi(\theta,\Sigma)\subset M\setminus\p\Sigma$ is a global surface of section. By construction, for a given $y\in\Phi(\theta,\Sigma)$ there exists $x\in\mathring\Sigma$ such that $y=\phi_X^{\theta\tau_x}(x)$. Then we have
\bean
\rho(y)&=\rho\circ\phi_X^{\theta\tau_x}(x)\\
&=\phi_X^{-\theta\tau_x}\circ \rho(x)\\
&=\phi_X^{(1-\theta)\tau_x}\circ\phi_X^{-\tau_x}\circ \rho(x)\\
&=\phi_X^{(1-\theta)\tau_x}\circ \rho\circ \phi_X^{\tau_x}(x).
\eea
Since $\rho\circ\phi_X^{\tau_x}(x)\in\mathring\Sigma$ and $\tau_x=\tau_{\rho\circ\phi_X^{\tau x}(x)}$, we conclude that 
$$
\rho\circ \Phi(\theta,\Sigma)=\Phi(1-\theta,\Sigma).
$$ 
In particular, the global surfaces of section $\Phi(0,\Sigma)$ and $\Phi(\frac{1}{2},\Sigma)$ are invariant under $\rho$.
\end{Rmk}


\end{document}